\documentclass[a4paper,10pt]{article}
\usepackage[latin1]{inputenc}
\usepackage[T1]{fontenc}
\usepackage{a4wide}
\usepackage{empheq}
\usepackage{amssymb,amsthm}
\usepackage{enumerate,bbm}

\newcommand{\m}{\mbox}
\newcommand{\eps}{\varepsilon}
\renewcommand{\leq}{\leqslant}
\renewcommand{\geq}{\geqslant}
\renewcommand{\Re}{\mathrm{Re}\,}

\newcommand{\Eq}{\Longleftrightarrow}
\newcommand{\vers}{\longrightarrow}

\newcommand{\tend}{\xrightarrow[n\to\infty]{}}
\newcommand{\cvf}{\rightharpoonup}

\newcommand{\lent}{[\kern-0.15em[}
\newcommand{\rent}{]\kern-0.15em]}

\newcommand{\ninf}[1]{{\left\|#1\right\|}_{\infty}}
\newcommand{\nh}[1]{{\left\|#1\right\|}_{H^1}}
\newcommand{\nht}[1]{{\left\|#1\right\|}_{H^{\frac{3}{4}}}}
\newcommand{\nlu}[1]{{\left\|#1\right\|}_{L^1}}
\newcommand{\nld}[1]{{\left\|#1\right\|}_{L^2}}
\newcommand{\nli}[1]{{\left\|#1\right\|}_{L^{\infty}}}
\newcommand{\nhs}[1]{{\left\|#1\right\|}_{H^s}}
\newcommand{\nluld}[1]{{\|#1\|}_{L_x^1L_{[T,+\infty)}^2}}
\newcommand{\nlcld}[1]{{\|#1\|}_{L_x^5L_{[T,+\infty)}^{10}}}

\newcommand{\Nlcld}[1]{{\left\|#1\right\|}_{L_x^5L_{[T,+\infty)}^{10}}}

\newcommand{\carre}[1]{{#1}^2}
\newcommand{\Carre}[1]{{\left( #1 \right)}^2}
\newcommand{\puiss}[2]{{#1}^{#2}}
\newcommand{\Puiss}[2]{{\left( #1 \right)}^{#2}}

\newcommand{\Qp}{Q^{\frac{p+1}{2}}}
\newcommand{\supt}{\sup_{t\in [0,T]}}
\newcommand{\drondx}{\frac{\partial}{\partial x}}
\newcommand{\abc}{\alpha,\beta,\gamma}
\newcommand{\abcd}{\alpha,\beta,\gamma,\delta}
\newcommand{\sig}[1]{\sigma_{#1}^{\lambda}}
\newcommand{\alphap}{\alpha_+^A}
\newcommand{\alpham}{\alpha_-^A}
\newcommand{\nls}{_{\mathrm{NLS}}}
\newcommand{\gn}{_{\mathrm{GN}}}
\newcommand{\wt}[1]{\widetilde{#1}}

\newcommand{\N}{\mathbb{N}}

\newcommand{\R}{\mathbb{R}}
\newcommand{\C}{\mathbb{C}}

\renewcommand{\H}{\mathcal{H}}
\renewcommand{\L}{\mathcal{L}}
\newcommand{\M}{\mathcal{M}}
\newcommand{\U}{\mathcal{U}}
\newcommand{\V}{\mathcal{V}}
\newcommand{\Y}{\mathcal{Y}}
\newcommand{\Z}{\mathcal{Z}}

\theoremstyle{plain}
\newtheorem{theo}{Theorem}[section]
\newtheorem*{theobis}{Theorem}
\newtheorem{prop}[theo]{Proposition}
\newtheorem{lem}[theo]{Lemma}
\newtheorem{cor}[theo]{Corollary}
\newtheorem{claim}[theo]{Claim}
\theoremstyle{definition}
\newtheorem{defi}[theo]{Definition}
\newtheorem{rem}[theo]{Remark}
\newtheorem*{notation}{Notation}
\newtheorem*{ack}{Acknowledgements}

\numberwithin{equation}{section}

\title{Construction and characterization of solutions converging to solitons for supercritical gKdV equations}
\author{Vianney Combet}
\date{Universit\'e de Versailles Saint-Quentin-en-Yvelines,
 Math\'ematiques, UMR 8100, \\
 45, av. des \'Etats-Unis,
 78035 Versailles Cedex, France\\
 vianney.combet@math.uvsq.fr}

\begin{document}

\maketitle

\begin{abstract}
We consider the generalized Korteweg-de Vries equation \[ \partial_t u +\partial_x^3 u +\partial_x(u^p)=0,\quad (t,x)\in\R^2, \] in the supercritical case $p>5$, and we are interested in solutions which converge to a soliton in large time in $H^1$. In the subcritical case ($p<5$), such solutions are forced to be exactly solitons by variational characterization \cite{bss,weinstein:lyapunov}, but no such result exists in the supercritical case. In this paper, we first construct a "special solution" in this case by a compactness argument, \emph{i.e.} a solution which converges to a soliton without being a soliton. Secondly, using a description of the spectrum of the linearized operator around a soliton \cite{pegoweinstein}, we construct a one parameter family of special solutions which characterizes all such special solutions.
\end{abstract}

\section{Introduction}

\subsection{The generalized Korteweg-de Vries equation} \label{subsec:gKdV}

We consider the generalized Korteweg-de Vries equation: \begin{equation} \label{eq:gKdV} \tag{gKdV} \begin{cases} \partial_t u+\partial_x^3 u +\partial_x(u^p)=0\\ u(0)=u_0\in H^1(\R) \end{cases} \end{equation} where $(t,x)\in\R^2$ and $p\geq 2$ is integer. The following quantities are formally conserved for solutions of \eqref{eq:gKdV}: \begin{gather} \label{eq:mass} \int u^2(t) = \int u^2(0)\quad \m{(mass)},\\ \label{eq:energy} E(u(t)) = \frac{1}{2}\int u_x^2(t) -\frac{1}{p+1}\int u^{p+1}(t) = E(u(0))\quad \m{(energy)}. \end{gather}

Kenig, Ponce and Vega \cite{kpv} have shown that the local Cauchy problem for \eqref{eq:gKdV} is well posed in $H^1(\R)$: for $u_0\in H^1(\R)$, there exist $T>0$ and a solution $u\in C^0([0,T],H^1(\R))$ of \eqref{eq:gKdV} satisfying $u(0)=u_0$ which is unique in some class $Y_T\subset C^0([0,T],H^1(\R))$. Moreover, if $T^*\geq T$ is the maximal time of existence of $u$, then either $T^*=+\infty$ which means that $u(t)$ is a global solution, or $T^*<+\infty$ and then $\nh{u(t)} \to +\infty$ as $t\uparrow T^*$ ($u(t)$ is a finite time blow up solution). Throughout this paper, when referring to an $H^1$ solution of \eqref{eq:gKdV}, we mean a solution in the above sense. Finally, if $u_0\in H^s(\R)$ for some $s\geq 1$, then $u(t)\in H^s(\R)$ for all $t\in [0,T)$.

In the case where $2\leq p<5$, it is standard that all solutions in $H^1$ are global and uniformly bounded by the energy and mass conservations and the following Gagliardo-Nirenberg inequality: \begin{equation} \label{eq:GN} \forall v\in H^1(\R),\quad \int \puiss{|v|}{p+1} \leq C\gn(p) \Puiss{\int v_x^2}{\frac{p-1}{4}}\Puiss{\int v^2}{\frac{p+3}{4}} \end{equation} with optimal constant $C\gn(p)>0$. In the case $p=5$, the existence of finite time blow up solutions was proved by Merle \cite{merle} and Martel and Merle \cite{martel:blowup}. Therefore $p=5$ is the critical exponent for the long time behavior of solutions of \eqref{eq:gKdV}. For $p>5$, the existence of blow up solutions is an open problem.

We recall that a fundamental property of equations \eqref{eq:gKdV} is the existence of a family of explicit traveling wave solutions. Let $Q$ be the only solution (up to translations) of \[ Q>0,\quad Q\in H^1(\R),\quad Q''+Q^p=Q,\quad \m{i.e.}\ Q(x)=\Puiss{\frac{p+1}{2\cosh^2\left(\frac{p-1}{2}x\right)}}{\frac{1}{p-1}}. \] Note that $Q$ is the unique minimizer of the Gagliardo-Nirenberg inequality \eqref{eq:GN} (see \cite{cazenave} for the case $p=5$ for example), \emph{i.e.} for $v\in H^1(\R)$: \begin{equation} \label{eq:caracGN} {\|v\|}_{L^{p+1}}^{p+1} = C\gn(p)\nld{v_x}^{\frac{p-1}{2}}\nld{v}^{\frac{p+3}{2}} \Eq \exists (\lambda_0,a_0,b_0)\in\R_+^* \times\R\times\R :  v(x)=a_0Q(\lambda_0 x+b_0). \end{equation} For all $c_0>0$ and $x_0\in\R$, $R_{c_0,x_0}(t,x)=Q_{c_0}(x-x_0-c_0t)$ is a solution of \eqref{eq:gKdV}, where \[ Q_{c_0}(x)=c_0^{\frac{1}{p-1}}Q(\sqrt{c_0}x). \] We call solitons these solutions though they are known to be solitons only for $p=2,3$ (in the sense that they are stable by interaction).

It is well known that solitons are orbitally stable (see definition \ref{th:orbital}) for $p<5$ and unstable for $p>5$. An important fact used by Weinstein in \cite{weinstein:lyapunov} to prove their orbital stability when $p<5$ is the following variational characterization of $Q_{c_0}$: if $u$ is a solution of \eqref{eq:gKdV} such that $E(u)= E(Q_{c_0})$ and $\int u^2 = \int Q_{c_0}^2$ for some $c_0>0$, then there exists $x_0\in\R$ such that $u=R_{c_0,x_0}$. As a direct consequence, if now $u(t)$ is a solution such that \begin{equation} \label{eq:tendversQc} \lim_{t\to+\infty} \inf_{y\in\R} {\|u(t)-Q_{c_0}(\cdot-y)\|}_{H^1(\R)}=0 \end{equation} (\emph{i.e.} $u$ converges to $Q_{c_0}$ in the suitable sense), then $u=R_{c_0,x_0}$. For $p=5$, the same is true for similar reasons (see \cite{weinstein:singularities}).

In the present paper, we focus on the supercritical case $p>5$. Some asymptotic results around solitons have been proved: orbital instability of solitons by Bona \emph{et al.} \cite{bss} (see also \cite{gss}) and asymptotic stability (in some sense) by Martel and Merle \cite{martel:general} for example. But available variational arguments do not allow to classify all solutions of \eqref{eq:gKdV} satisfying \eqref{eq:tendversQc}. In fact, in section \ref{sec:compactness}, we construct a solution of \eqref{eq:gKdV} satisfying \eqref{eq:tendversQc} which is not a soliton (we call \emph{special solution} such a solution). In section \ref{sec:contraction}, by another method, we construct a whole family of such solutions, and we completely characterize solutions satisfying \eqref{eq:tendversQc}. This method is strongly inspired of arguments developed by Duyckaerts and Roudenko in \cite{duy}, themselves an adaptation of arguments developed by Duyckaerts and Merle in \cite{duymerle}. For reader's convenience, we recall in the next section the results in \cite{duy} related to our paper.

\subsection{The Non-Linear Schrödinger equation case}

We recall Duyckaerts and Roudenko's results for \eqref{eq:NLS}. They consider in \cite{duy} the 3d focusing cubic non-linear Schrödinger equation: \begin{equation} \label{eq:NLS} \tag{NLS} \begin{cases} i\partial_t u+ \Delta u + {|u|}^2u =0,\quad (x,t)\in\R^3\times\R,\\ u_{|t=0} =u_0\in H^1(\R^3). \end{cases} \end{equation} This equation is $\dot{H}^{1/2}$-critical, and so $L^2$-supercritical like \eqref{eq:gKdV} for $p>5$, while \cite{duymerle} is devoted to the $\dot{H}^1$-critical equation. Similarly to \eqref{eq:gKdV}, \eqref{eq:NLS} is locally well posed in $H^1$, and solutions of \eqref{eq:NLS} satisfy the following conservation laws: \begin{gather*} E\nls [u](t) = \frac{1}{2} \int {|\nabla u(x,t)|}^2\,dx -\frac{1}{4} \int {|u(x,t)|}^4\,dx = E\nls [u](0),\\ M\nls [u](t) = \int {|u(x,t)|}^2\,dx = M\nls [u](0). \end{gather*} Moreover, if $Q$ is the unique (in a suitable sense) solution of the non-linear elliptic equation $-Q+\Delta Q+{|Q|}^2 Q=0$, then $e^{it}Q(x)$ is a soliton solution of \eqref{eq:NLS}.

Theorem 2 in \cite{duy} states the existence of two radial solutions $Q^+(t)$ and $Q^-(t)$ of \eqref{eq:NLS} such that $M\nls[Q^+]=M\nls[Q^-]=M\nls[Q]$, $E\nls[Q^+]=E\nls[Q^-]=E\nls[Q]$, $[0,+\infty)$ is in the time domain of definition of $Q^{\pm}(t)$, and there exists $e_0>0$ such that: $\forall t\geq 0,\quad \nh{Q^{\pm}(t)-e^{it}Q}\leq Ce^{-e_0t}$. Moreover, $Q^-(t)$ is globally defined and scatters for negative time, and the negative time of existence of $Q^+(t)$ is finite.

They also prove the following classification theorem \cite[theorem 3]{duy}:

\begin{theobis}[\cite{duy}]
Let $u$ be a solution of \eqref{eq:NLS} satisfying $E\nls[u]M\nls[u]=E\nls[Q]M\nls[Q]$.
\begin{enumerate}[(a)]
\item If $\nld{\nabla u_0}\nld{u_0} < \nld{\nabla Q}\nld{Q}$, then either $u$ scatters or $u=Q^-$ up to the symmetries.
\item If $\nld{\nabla u_0}\nld{u_0} = \nld{\nabla Q}\nld{Q}$, then $u=e^{it}Q$ up to the symmetries.
\item If $\nld{\nabla u_0}\nld{u_0} > \nld{\nabla Q}\nld{Q}$ and $u_0$ is radial or of finite variance, then either the interval of existence of $u$ is of finite length or $u=Q^+$ up to the symmetries.
\end{enumerate}
\end{theobis}

\noindent In particular, if $\lim_{t\to+\infty} \nh{u(t)-e^{it}Q}=0$, then $u=e^{it}Q$, $Q^+$ or $Q^-$ up to the symmetries.

Among the various ingredients used to prove results above, one of the most important is a sharp analysis of the spectrum $\sigma(\L\nls)$ of the linearized Schrödinger operator around the ground state solution $e^{it}Q$, due to Grillakis \cite{grillakis} and Weinstein \cite{weinstein:modulational}. They prove that $\sigma(\L\nls)\cap\R=\{ -e_0,0,+e_0\}$ with $e_0>0$, and moreover that $e_0$ and $-e_0$ are simple eigenvalues of $\L\nls$ with eigenfunctions $\Y_+^{\mathrm{NLS}}$ and $\Y_-^{\mathrm{NLS}} = \overline{\Y_+^{\mathrm{NLS}}}$. This structure, which is similar for \eqref{eq:gKdV} according to Pego and Weinstein \cite{pegoweinstein}, will also be crucial to prove our main result (exposed in the next section).

\subsection{Main result and outline of the paper}

In this paper, we consider similar questions for the \eqref{eq:gKdV} equation in the supercritical case $p>5$. Recall that similarly to the \eqref{eq:NLS} case, Pego and Weinstein have determined in \cite{pegoweinstein} the spectrum of the linearized operator $\L$ around the soliton $Q(x-t)$: $\sigma(\L)\cap\R = \{-e_0,0,+e_0\}$ with $e_0>0$, and moreover $e_0$ and $-e_0$ are simple eigenvalues of $\L$ with eigenfunctions $\Y_+$ and $\Y_-$ which are exponentially decaying (see proposition \ref{th:spectrum} and corollary \ref{th:Ydecay}). We now state precisely our main result:

\begin{theo} \label{th:main}
Let $p>5$.
\begin{enumerate}
\item \emph{(Existence of a family of special solutions).} There exists a one-parameter family ${(U^A)}_{A\in\R}$ of solutions of \eqref{eq:gKdV} such that \[ \lim_{t\to +\infty} \nh{U^A(t,\cdot+t)-Q}=0. \] Moreover, for all $A\in\R$, there exists $t_0=t_0(A)\in\R$ such that for all $s\in\R$, there exists $C>0$ such that \[ \forall t\geq t_0,\quad \nhs{U^A(t,\cdot+t)-Q-Ae^{-e_0t}\Y_+}\leq Ce^{-2e_0 t}. \]
\item \emph{(Classification of special solutions).} If $u$ is a solution of \eqref{eq:gKdV} such that \[ \lim_{t\to+\infty} \inf_{y\in\R} \nh{u(t)-Q(\cdot-y)} =0, \] then there exist $A\in\R$, $t_0\in\R$ and $x_0\in\R$ such that $u(t)=U^A(t,\cdot-x_0)$ for $t\geq t_0$.
\end{enumerate}
\end{theo}

\begin{rem}
From theorem \ref{th:main}, there are actually only three different special solutions $U^A$ up to translations in time and in space: $U^1$, $U^{-1}$ and $Q(\cdot-t)$ (see proposition \ref{th:troistypes}). This is of course related to the three solutions of \eqref{eq:NLS} constructed in \cite{duy}: $Q^+(t)$, $Q^-(t)$ and $e^{it}Q$.

From section \ref{subsec:remarks}, we can chose the normalization of $\Y_{\pm}$ so that for $A<0$, $\nld{\partial_x U^A}<\nld{Q'}$. Then $U^{-1}(t)$ is global, \emph{i.e.} defined for all $t\in\R$. It would be interesting to investigate in more details its behavior as $t\to-\infty$. On the other hand, the behavior of $U^1(t)$ is not known for $t<t_0$.
\end{rem}

\begin{rem}
By scaling, theorem \ref{th:main} extends to $Q_c$ for all $c>0$ (see corollary \ref{th:uac} at the end of the paper).
\end{rem}

The paper is organized as follows. In the next section we recall some properties of the solitons, and in particular we recall the proof of their orbital instability when $p>5$. This result is well known \cite{bss}, but our proof with an explicit initial data is useful to introduce some suitable tools to the study of solitons of \eqref{eq:gKdV} (as modulation, Weinstein's functional, monotonicity, linearized equation, etc.). Moreover, it is the first step to construct \emph{one} special solution in section \ref{sec:compactness} by compactness, similarly as Martel and Merle \cite{martel:general}. This proof does not use the precise analysis of the spectrum of $\L$ due to Pego and Weinstein \cite{pegoweinstein}, and so can be hopefully adapted to equations for which the spectrum of the linearized operator is not well known. To fully prove theorem \ref{th:main} (existence and uniqueness of a family of special solutions, section \ref{sec:contraction}), we rely on the method introduced in \cite{duymerle} and \cite{duy}.

\begin{ack}
The author would like to thank Nikolay Tzvetkov for suggesting the problem studied in this work, and for pointing out to him reference \cite{pegoweinstein}. He would also like to thank Luc Robbiano and Yvan Martel for their constructive remarks.
\end{ack}

\section{Preliminary results}

We recall here some well known properties of the solitons and some results of stability around the solitons. We begin by recalling notation and simple facts on the functions $Q(x)$ and $Q_c(x)=c^{\frac{1}{p-1}}Q(\sqrt cx)$ defined in section \ref{subsec:gKdV}.

\begin{notation}
They are available in the whole paper.
\begin{enumerate}[(a)]
\item $(\cdot,\cdot)$ denotes the $L^2(\R)$ scalar product, and $\bot$ the orthogonality with respect to $(\cdot,\cdot)$.
\item The Sobolev space $H^s$ is defined by $H^s(\R) = \{ u\in \mathcal{D}'(\R)\ |\ {(1+\xi^2)}^{s/2}\hat{u}(\xi) \in L^2(\R) \}$, and in particular $H^1(\R) = \{ u\in L^2(\R)\ |\ \nh{u}^2 = \nld{u}^2 +\nld{u'}^2 <+\infty \} \hookrightarrow L^{\infty}(\R)$.
\item We denote $\drondx v = \partial_x v = v_x$ the partial derivative of $v$ with respect to $x$, and $\partial_x^s = \partial^s$ the $s$-order partial derivative with respect to $x$ when no confusion is possible.
\item All numbers $C,K$ appearing in inequalities are real constants (with respect to the context) strictly positive, which may change in each step of an inequality.
\end{enumerate}
\end{notation}

\begin{claim} \label{th:solitons}
For all $c>0$, one has:
\begin{enumerate}[(i)]
\item $Q_c>0$, $Q_c$ is even, $Q_c$ is $C^{\infty}$, and $Q'_c(x)<0$ for all $x>0$.
\item There exist $K_1,K_2>0$ such that: $\forall x\in\R,\quad K_1e^{-\sqrt{c}|x|}\leq Q_c(x)\leq K_2e^{-\sqrt{c}|x|}$.
\item There exists $C_p>0$ such that for all $j\geq 0$, $Q_c^{(j)}(x)\sim C_p e^{-\sqrt c|x|}$ when $|x|\to+\infty$.

    In particular, for all $j\geq 1$, there exists $C_j>0$ such that: $\forall x\in\R,\quad |Q_c^{(j)}(x)|\leq C_j e^{-\sqrt{c}|x|}$.
\item The following identities hold: \begin{equation} \label{eq:Qcmass} \int Q_c^2 = c^{\frac{5-p}{2(p-1)}} \int Q^2\quad ,\quad \int \Carre{Q'_c} = c^{\frac{p+3}{2(p-1)}}\int Q'^2. \end{equation}
\end{enumerate}
\end{claim}

\subsection{Weinstein's functional linearized around $Q$} \label{subsec:linear}

We introduce here the Weinstein's functional $F$ and give an expression of $F(Q+a)$ for $a$ small which will be very useful in the rest of the paper. We recall first that the energy of a function $\varphi\in H^1$ is defined by $E(\varphi) = \frac{1}{2}\int (\partial_x\varphi)^2-\frac{1}{p+1}\int \varphi^{p+1}$.

\begin{defi}
Weinstein's functional is defined for $\varphi\in H^1$ by $F(\varphi) = E(\varphi)+\frac{1}{2}\int \varphi^2$.
\end{defi}

\begin{claim}
If $u_0\in H^1$ and $u(t)$ solves \eqref{eq:gKdV} with $u(0)=u_0$, then for all $t\in [0,T^*)$, $F(u(t))=F(u_0)$. It is an immediate consequence of \eqref{eq:mass} and \eqref{eq:energy}.
\end{claim}

\begin{lem}[Weinstein's functional linearized around $Q$] \label{th:Flin}
For all $C>0$, there exists $C'>0$  such that, for all $a\in H^1$ verifying $\nh{a}\leq C$, \begin{equation} \label{eq:Flin} F(Q+a) = F(Q)+\frac{1}{2}(La,a)+K(a)\end{equation} where $La=-\partial_x^2 a+a-pQ^{p-1}a$, and $K : H^1\to\R$ satisfies $|K(a)|\leq C'\nh{a}^3$.
\end{lem}

\begin{proof}
Let $a\in H^1$ be such that $\nh{a}\leq C$. Then we have \begin{align*} E(Q+a) &= \frac{1}{2}\int \carre{(Q'+\partial_x a)} -\frac{1}{p+1}\int \puiss{(Q+a)}{p+1}\\ &= E(Q)+ \frac{1}{2}\int \Carre{\partial_x a} +\int Q'\cdot \partial_x a -\frac{1}{p+1}\int \left[ (p+1)Q^p a + \frac{(p+1)p}{2}Q^{p-1} a^2 +R(a)\right]\\ &= E(Q)+ \frac{1}{2}\int \Carre{\partial_x a} -\int Q a -\frac{p}{2}\int Q^{p-1}a^2 -\frac{1}{p+1}\int R(a) \end{align*} since $Q''+Q^p=Q$, and where $R(a) = \sum_{k=3}^{p+1} \binom{p+1}{k} Q^{p+1-k}a^k$. Since $\ninf{a}\leq C\nh{a}\leq C$, then $|R(a)|\leq C|a|^3\leq C\ninf{a}|a|^2$, and so $K(a)=-\frac{1}{p+1}\int R(a)$ verifies $|K(a)|\leq C'\nh{a}^3$. Moreover, we have more simply: $\int \carre{(Q+a)} = \int Q^2 +\int a^2 +2\int Qa$. Finally we have \[F(Q+a) = F(Q)+ \frac{1}{2}\int a^2+ \frac{1}{2}\int \Carre{\partial_x a} -\frac{p}{2}\int Q^{p-1}a^2 +K(a). \qedhere \]
\end{proof}

\begin{claim}[Properties of $L$] \label{th:L}
The operator $L$ defined in lemma \ref{th:Flin} is self-adjoint and satisfies the following properties:
\begin{enumerate}[(i)]
\item First eigenfunction: $L\Qp = -\lambda_0\Qp$ where $\lambda_0=\frac{1}{4}(p-1)(p+3)>0$.
\item Second eigenfunction: $LQ' = 0$, and $\ker L = \{ \lambda Q'\ ;\ \lambda\in\R\}$.
\item Scaling: If we denote $S = {\left. \frac{dQ_c}{dc} \right|}_{c=1}$, then $S(x) = \frac{1}{p-1}Q(x) +\frac{1}{2}xQ'(x)$ and $LS = -Q$.
\item Coercivity: There exists $\sigma_0>0$ such that for all $u\in H^1(\R)$ verifying $(u,Q')= (u,\Qp)=0$, one has $(Lu,u)\geq \sigma_0 \nld{u}^2$.
\end{enumerate}
\end{claim}

\begin{proof}
The first three properties follow from straightforward computation, except for $\ker L$ which can be determined by ODE techniques, see \cite[proposition 2.8]{weinstein:modulational}. The property of coercivity follows easily from (i), (ii) and classical results on self-adjoint operators and Sturm-Liouville theory.
\end{proof}

\begin{lem} \label{th:coercif}
There exist $K_1,K_2>0$ such that for all $\eps\in H^1$ verifying $\eps\bot Q'$: \[(L\eps,\eps) = \int\eps_x^2 +\int\eps^2 -p\int Q^{p-1}\eps^2 \geq K_1\nh{\eps}^2 -K_2\Carre{\int \eps \Qp}. \]
\end{lem}

\begin{proof}
By claim \ref{th:L}, we already know that there exists $\sigma_0>0$ such that for all $\eps$ satisfying $\eps\bot\Qp$ and $\eps\bot Q'$, we have $(L\eps,\eps)\geq \sigma_0 \nld{\eps}^2$. The first step is to replace the $L^2$ norm by the $H^1$ one in this last inequality, which is easy if we choose $\sigma_0$ small enough. If we do not suppose $\eps\bot\Qp$, we write $\eps=\eps_1+a\Qp$ with $a=(\int\eps\Qp) \Puiss{\int Q^{p+1}}{-1}$ such that $\eps_1\bot\Qp$ for the $L^2$ scalar product, but also for the bilinear form $(L\cdot,\cdot)$ since $\Qp$ is an eigenvector for $L$. Since $\Qp\bot Q'$, we obtain easily the desired inequality from the previous step.
\end{proof}

\subsection{Orbital stability and decomposition of a solution around $Q$}

In this paper, we consider only solutions which stay close to a soliton. So it is important to define properly this notion, and the invariance by translation leads us to consider for $\eps>0$ the "tube" \[U_{\eps} = \{ u\in H^1\ |\ \inf_{y\in\R} \nh{u-Q_c(\cdot-y)} \leq\eps\}. \]

\begin{defi} \label{th:orbital}
The solitary wave $Q_c$ is (orbitally) \emph{stable} if and only if for every $\eps>0$, there exists $\delta>0$ such that if $u_0\in U_{\delta}$, then the associated solution $u(t)\in U_{\eps}$ for all $t\in\R$. The solitary wave $Q_c$ is \emph{unstable} if $Q_c$ is not stable.
\end{defi}

\begin{theo}
$Q_c$ is stable if and only if $p<5$.
\end{theo}

\begin{rem}
\begin{enumerate}
\item This theorem is proved by Bona \emph{et al.} \cite{bss} for $p\neq 5$ and by Martel and Merle \cite{martel:instability} for $p=5$. Nevertheless, we give an explicit proof of the instability of $Q$ when $p>5$ (\emph{i.e.} we exhibit an explicit sequence of initial data which contradicts the stability) which will be useful to construct the special solution by the compactness method (section \ref{sec:compactness}).
\item An important ingredient to prove this theorem is the following lemma of modulation close to $Q$. Its proof is based on the implicit function theorem (see for example \cite[lemma 4.1]{bss} for details). The orthogonality to $Q'$ obtained by this lemma will be of course useful to exploit the coercivity of the bilinear form $(L\cdot,\cdot)$. Finally, we conclude this section by a simple but useful lemma which describes the effect of small translations on $Q$.
\end{enumerate}
\end{rem}

\begin{lem}[Modulation close to $Q$] \label{th:decomp}
There exist $\eps_0>0$, $C>0$ and a unique $C^1$ map $\alpha~: U_{\eps_0}\vers\R$ such that for every $u\in U_{\eps_0}$, $\eps=u(\cdot+\alpha(u))-Q$ verifies \[ (\eps,Q')=0 \m{ and } \nh{\eps}\leq C\inf_{y\in\R} \nh{u-Q(\cdot-y)}\leq C\eps_0.\]
\end{lem}

\begin{lem} \label{th:Qmass}
There exist $h_0>0$, $A_0>0$ and $\beta>0$ such that: \begin{enumerate}[(i)] \item if $|h|\leq h_0$ then $\beta h^2 \leq \nh{Q-Q(\cdot+h)}^2 \leq 4\beta h^2$, \item if $|h|>h_0$ then $\nh{Q-Q(\cdot+h)}^2> A_0$. \end{enumerate}
\end{lem}

\begin{proof}
It is a simple application of Taylor's theorem to $f$ defined by $f(a)=\nh{Q-Q(\cdot+a)}^2$.
\end{proof}

\subsection{Instability of $Q$ for $p>5$}

In this section, we construct an explicit sequence $(u_{0,n})_{n\geq 1}$ of initial data which contradicts the stability of $Q$:

\begin{prop} \label{th:uzero}
Let $u_{0,n}(x)=\lambda_nQ(\lambda_n^2x)$ with $\lambda_n=1+\frac{1}{n}$ for $n\geq 1$. Then \begin{equation} \label{eq:uon} \int u_{0,n}^2 = \int Q^2\quad,\quad E(u_{0,n})<E(Q)\quad \m{ and }\quad \nh{u_{0,n}-Q}\tend 0. \end{equation}
\end{prop}

\begin{proof}
The first and the last facts are obvious thanks to substitutions and the dominated convergence theorem. For the energy inequality, we compute $E(u_{0,n}) = \frac{\lambda_n^4}{2}\int Q'^2 - \frac{\lambda_n^{p-1}}{p+1}\int Q^{p+1}$. But $2\int Q'^2 = \frac{p-1}{p+1} \int Q^{p+1}$ by Pohozaev identities, and so \begin{align*} E(u_{0,n})-E(Q) &= \left[ \frac{p-1}{4}\times(\lambda_n^4-1)-(\lambda_n^{p-1}-1)\right]\cdot \frac{1}{p+1}\int Q^{p+1}\\ &= \left[\sum_{k=2}^4 \left\{ \frac{p-1}{4}\binom{4}{k}-\binom{p-1}{k} \right\}\frac{1}{n^k} - \sum_{k=5}^{p-1} \binom{p-1}{k} \frac{1}{n^k} \right] \cdot \frac{1}{p+1}\int Q^{p+1}. \end{align*} To conclude, it is enough to show that $\binom{p-1}{k}>\frac{p-1}{4}\binom{4}{k}$ for $k\in\{2,3,4\}$, which is equivalent to show that $\binom{p-2}{k-1} = \frac{k}{p-1}\binom{p-1}{k} > \frac{k}{4}\binom{4}{k} = \binom{3}{k-1}$, which is right since $p>5$ and $k>1$.
\end{proof}

\begin{rem} \label{th:choixuon}
We do not really need to know the explicit expression of $u_{0,n}$ to prove the instability of $Q$: initial data satisfying conditions \eqref{eq:uon} and decay in space would fit. For example, we could have chosen $\lambda_n=1-\frac{1}{n}$, so that conditions \eqref{eq:uon} hold for $n$ large (in fact $E(u_{0,n})-E(Q) \sim \frac{(p-1)(5-p)}{2(p+1)}\int Q^{p+1}\cdot\frac{1}{n^2}<0$ as $n\to+\infty$ in this case).
\end{rem}

\begin{theo} \label{th:2}
Let $u_n$ be the solution associated to $u_{0,n}$ defined in proposition \ref{th:uzero}. Then \begin{equation} \label{eq:th2} \exists\delta>0, \forall n\geq 1, \exists T_n\in\R_+ \m{ such that } \inf_{y\in\R} \nh{u_n(T_n)-Q(\cdot-y)}>\delta. \end{equation}
\end{theo}

\bigskip

\noindent $\bullet$ We prove this theorem by contradiction, \emph{i.e.} we suppose: \[\forall\eps>0, \exists n_0\geq 1, \forall t\in\R_+, \inf_{y\in\R} \nh{u_{n_0}(t)-Q(\cdot-y)}\leq\eps, \] and we apply this assumption to $\eps_0$ given by lemma \ref{th:decomp}. Dropping $n_0$ for a while, the situation amounts in: \[\int u_0^2 = \int Q^2\ ,\ E(u_0)<E(Q)\ \m{ and }\ \forall t\in\R_+, \inf_{y\in\R} \nh{u(t)-Q(\cdot-y)}\leq\eps_0.\] The last fact implies that $u(t)\in U_{\eps_0}$ for all $t\in\R_+$, so lemma \ref{th:decomp} applies and we can define $x(t)=\alpha(u(t))$ which is $C^1$ by standard arguments (see \cite{martel:instability} for example), and $\eps(t,x)=u(t,x+x(t))-Q(x)$ which verifies $(\eps(t),Q')=0$ and $\nh{\eps(t)}\leq C\eps_0$ for all $t\in\R_+$. Note that $x(t)$ is usually called the \emph{center of mass} of $u(t)$. Before continuing the proof, we give the equation verified by $\eps$ and an interesting consequence on $x'$.

\begin{prop} \label{th:eqeps}
There exists $C>0$ such that \[\eps_t -(L\eps)_x = (x'(t)-1)(Q+\eps)_x +R(\eps), \] where $\nlu{R(\eps(t))}\leq C\nh{\eps(t)}^2$. As a consequence, one has: $|x'(t)-1|\leq C\nh{\eps(t)}$.
\end{prop}

\begin{proof}
Since $u(t,x)=Q(x-x(t))+\eps(t,x-x(t))$ by definition of $\eps$ and $-\partial_tu = \partial_x^3u+\partial_x(u^p)$, we obtain \[x'(t)(Q+\eps)_x -\eps_t = Q_{xxx}+\eps_{xxx} +(Q^p)_x +p(Q^{p-1}\eps)_x +R(\eps)\] where \[ R(\eps) =\drondx\left( \sum_{k=2}^p \binom{p}{k} Q^{p-k}\eps^k \right)= \sum_{k=2}^p \binom{p}{k} \left[ (p-k)Q'Q^{p-k-1}\eps^k + kQ^{p-k}\eps_x\eps^{k-1}\right]. \] As $\ninf{\eps}\leq C\nh{\eps} \leq C\eps_0$, we have $|R(\eps)|\leq C|\eps|^2+C'|\eps_x\eps|$, and so $R(\eps)$ is such as expected. Moreover, since $La=-a_{xx}+a-pQ^{p-1}a$ and $Q''+Q^p=Q$, we get \[ -\eps_t-\eps_{xxx}-p(Q^{p-1}\eps)_x = Q_{xxx} +(Q^p)_x-x'(t)(Q+\eps)_x+R(\eps) \] and so $-\eps_t +(L\eps)_x = Q_x -x'(t)(Q+\eps)_x +\eps_x+R(\eps)$.

To obtain the estimate on $x'$, we multiply the equation previously found by $Q'$ and integrate. Since $(\eps_t,Q')= (\eps,Q')_t = 0$, it gives with an integration by parts: \[\int (L\eps)Q'' = (x'-1)\int (Q'^2+\eps_xQ')+\int R(\eps)Q'. \] Since $L$ is self-adjoint, we can write $(x'-1)\int (Q'^2+\eps_xQ') = \int (LQ'')\eps - \int R(\eps)Q'$. Now, from $\left| \int \eps_xQ' \right| \leq \nld{\eps_x}\nld{Q'}\leq \nh{\eps}\nld{Q'} \leq C\eps_0\nld{Q'}$, we choose $\eps_0$ small enough so that the last quantity is smaller than $\frac{1}{2}\int Q'^2$; and so we have \[|x'-1| \leq \frac{2}{\int Q'^2}\left( \left|\int (LQ'')\eps\right| + \left|\int R(\eps)Q'\right| \right). \] As $LQ''\in L^2(\R)$ and $Q'\in L^{\infty}(\R)$, then following the estimate on $R(\eps)$, we obtain the desired inequality by the Cauchy-Schwarz inequality.
\end{proof}

\bigskip

\noindent $\bullet$ Return to the proof of theorem \ref{th:2} and now consider \[\zeta(x)=\int_{-\infty}^x \left( S(y)+\beta \Qp(y)\right)dy\] for $x\in\R$, where $S$ is defined in claim \ref{th:L} and $\beta$ will be chosen later. We recall that $S(x)=\frac{1}{p-1}Q(x) +\frac{1}{2}xQ'(x)$ verifies $LS=-Q$, and in particular $S(x)=o(e^{-|x|/2})$ when $|x|\to+\infty$, since $Q(x),Q'(x)\sim C_p e^{-|x|}$ (see claim \ref{th:solitons}). By integration, we have $\zeta(x)=o(e^{x/2})$ when $x\to-\infty$, and $\zeta$ is bounded on $\R$.\bigskip

Now, the main idea of the proof is to consider the functional, defined for $t\in\R_+$, \[J(t)=\int \eps(t,x)\zeta(x)\,dx. \] The first step is to show that $J$ is defined and bounded in time thanks to the following proposition of decay properties of the solutions, and the second one is to show that $|J'|$ has a strictly positive lower bound, which will reach the desired contradiction. Firstly, if we choose $\eps_0$ small enough, we obtain the following proposition.

\begin{prop} \label{th:expdecay}
There exists $C>0$ such that for all $t\geq 0$ and $x_0>0$, \begin{equation} \label{eq:expdecay} \int_{x>x_0} (u^2+u_x^2)(t,x+x(t))\,dx \leq Ce^{-x_0/4}. \end{equation}
\end{prop}

\begin{rem} \label{th:monotonie}
Inequality \eqref{eq:expdecay} holds for all solution $u_n$ of \eqref{eq:gKdV} associated to the initial data $u_{0,n}$ defined in proposition \ref{th:uzero}, with $C>0$ independent of $n$. Indeed, we have $u=u_{n_0}$ for some $n_0\geq 1$, but the following proof shows that the final constant $C$ does not depend of $n_0$.
\end{rem}

\begin{proof}
It is based on the exponential decay of the initial data, and on monotonicity results that the reader can find in \cite[lemma 3]{martel:revisited}. We recall here their notation and their lemma of monotonicity.
\begin{itemize}
\item[$\diamond$] Let $\psi(x)=\frac{2}{\pi}\arctan(\exp(x/4))$, so that $\psi$ is increasing, $\lim_{-\infty}\psi=0$, $\psi(0)=\frac{1}{2}$,  $\lim_{+\infty} \psi=1$, $\psi(-x)=1-\psi(x)$ for all $x\in\R$, and $\psi(x)\sim Ce^{x/4}$ when $x\to -\infty$. Now let $x_0>0$, $t_0>0$ and define for $0\leq t\leq t_0$: $\psi_0(t,x)=\psi(x-x(t_0)+\frac{1}{2}(t_0-t)-x_0)$ and \[ \left\{\begin{aligned} I_{x_0,t_0}(t) &= \int u^2(t,x)\psi_0(t,x)\,dx,\\ J_{x_0,t_0}(t) &= \int (u_x^2+u^2-\frac{2}{p+1}u^{p+1})(t,x)\psi_0(t,x)\,dx. \end{aligned} \right.\] Then, if we choose $\eps_0$ small enough, there exists $K>0$ such that for all $t\in [0,t_0]$, we have \[ \left\{ \begin{aligned} I_{x_0,t_0}(t_0)-I_{x_0,t_0}(t) &\leq K\exp\left(-\frac{x_0}{4}\right),\\ J_{x_0,t_0}(t_0)-J_{x_0,t_0}(t) &\leq K\exp\left(-\frac{x_0}{4}\right). \end{aligned} \right. \]
\item[$\diamond$] Now, let us prove how this result can preserve the decay of the initial data to the solution for all time, \emph{on the right} (which means for $x>x_0$ for all $x_0>0$). If we apply it to $t=0$ and replace $t_0$ by $t$, we obtain for all $t>0$: \begin{multline*} \int (u_x^2+u^2)(t,x+x(t))\psi(x-x_0)\,dx\\ \leq C'\int (u^2_{0x}+u_0^2)(x)\psi(x-x(t)+\frac{1}{2}t-x_0)\,dx +K'e^{-x_0/4}. \end{multline*} But by proposition \ref{th:eqeps}, we have $|x'-1|\leq C\nh{\eps}\leq C\eps_0$, thus if we choose $\eps_0$ small enough, we have $|x'-1|\leq \frac{1}{2}$, and so we obtain by the mean value inequality (notice that $x(0)=\alpha(u_{0,n_0})=0$): $|x(t)-t|\leq \frac{1}{2}t$. We deduce that $-x(t)+\frac{1}{2}t\leq 0$, and since $\psi$ is increasing, we obtain \[ \int (u_x^2+u^2)(t,x+x(t))\psi(x-x_0)\,dx \leq C\int (u^2_{0x}+u_0^2)(x)\psi(x-x_0)\,dx +Ke^{-x_0/4}. \]
\item[$\diamond$] Now we explicit exponential decay of $u_0$. In fact, we have clearly $(u^2_{0x}+u_0^2)(x)\sim Ce^{-2\lambda^2|x|}\leq Ce^{-2|x|}$ when $x\to\pm\infty$. Moreover, since $\psi(x)\leq Ce^{x/4}$ for all $x\in\R$, we have \begin{align*} \int (u^2_{0x}+u_0^2)(x)\psi(x-x_0)\,dx &\leq C\int (u^2_{0x}+u_0^2)(x)e^{\frac{x-x_0}{4}}\,dx\\ &\leq Ce^{-x_0/4} \int (u^2_{0x}+u_0^2)(x)e^{x/4}\,dx \leq C' e^{-x_0/4}. \end{align*}
\item[$\diamond$] Finally, we have more simply \[ \int (u_x^2+u^2)(t,x+x(t))\psi(x-x_0)\,dx \geq \frac{1}{2}\int_{x>x_0} (u_x^2+u^2)(t,x+x(t))\,dx, \] and so the desired inequality. \qedhere
\end{itemize}
\end{proof}

\noindent $\bullet$ Now this proposition is proved, we can easily show the first step of the proof of theorem \ref{th:2}.\bigskip

\emph{1st step:} We bound $|J(t)|$ independently of time by writing \[J(t)= \int \eps(t,x)\zeta(x)\,dx = \int_{x>0} \eps(t,x)\zeta(x)\,dx +\int_{x<0} \eps(t,x)\zeta(x)\,dx, \] so that \begin{align*} |J(t)| &\leq \ninf{\zeta}\int_{x>0} (Q(x)+|u(t,x+x(t))|)\,dx + \sqrt{\int_{x<0} \eps^2(t,x)\,dx}\sqrt{\int_{x<0} \zeta^2(x)\,dx}\\ &\leq \ninf{\zeta}\nlu{Q}+\ninf{\zeta}U + \nld{\eps(t)}V, \end{align*} where: \begin{enumerate}[i)] \item $\nld{\eps(t)}\leq \nh{\eps}\leq C\eps_0<+\infty$, \item $V^2 = \int_{x<0} \zeta^2(x)\,dx <+\infty$ since $\zeta^2(x)=o(e^x)$ when $x\to -\infty$, \item thanks to \eqref{eq:expdecay}, we finally conclude the first step with: \begin{align*} U &= \int_{x>0} |u(t,x+x(t))|\,dx = \sum_{n=0}^{+\infty} \int_n^{n+1} |u(t,x+x(t))|\,dx \leq \sum_{n=0}^{+\infty} \Puiss{\int_{x>n} u^2(t,x+x(t))\,dx}{1/2}\\ &\leq \nld{u(t,\cdot+x(t))} +\sum_{n=1}^{+\infty} \Puiss{\int_{x>n} u^2(t,x+x(t))\,dx}{1/2} \leq C\eps_0+\nld{Q} +C\sum_{n=1}^{+\infty} e^{-n/8} <+\infty. \end{align*} \end{enumerate}

\emph{2nd step:} We evaluate $J'$ by using proposition \ref{th:eqeps} and by integrating by parts: \begin{align*} J' &= \int \eps_t \zeta = \int (L\eps)_x \zeta +(x'-1)\int Q_x \zeta +(x'-1)\int \eps_x \zeta +\int R(\eps)\zeta\\ &= -\int \eps L(\zeta') -(x'-1)\int Q \zeta' -(x'-1)\int\eps \zeta' +\int R(\eps)\zeta\\ &= -\int \eps(L S+\beta L \Qp) -(x'-1)\int Q(S+\beta \Qp) -(x'-1)\int\eps \zeta' +\int R(\eps)\zeta. \end{align*} Now we take $\beta = -\frac{\int QS}{\int Q^{\frac{p+3}{2}}}$ so that the second integral is null. Note  that by (iv) of claim \ref{th:solitons}, \[ \frac{d}{dc} \int Q_c^2 = 2\int Q_c \frac{dQ_c}{dc} = \left( \frac{5-p}{2(p-1)}\right) c^{\frac{5-p}{2(p-1)}-1}\int Q^2 <0 \] since $p>5$, and so by taking $c=1$ we remark that $\beta>0$. Moreover, since $\Qp$ is an eigenvector for $L$ for an eigenvalue $-\lambda_0$ with $\lambda_0>0$ (see claim \ref{th:L}), we deduce \begin{align*} J' &= -\int \eps(-Q-\beta\lambda_0 \Qp) -(x'-1)\int\eps \zeta' +\int R(\eps)\zeta\\ &= \beta\lambda_0\int \eps \Qp +\int Q\eps -(x'-1)\int\eps \zeta' +\int R(\eps)\zeta. \end{align*} But for the last three terms, we remark that: \begin{enumerate}[a)] \item the mass conservation $\int u^2(t) = \int u_0^2$ implies that $\int Q^2 +2\int \eps Q +\int \eps^2 = \int Q^2$ and so $\left|\int Q\eps\right| \leq \frac{1}{2}\int \eps^2\leq \frac{1}{2} \nh{\eps}^2$, \item thanks to proposition \ref{th:eqeps}, we have $\left|-(x'-1)\int\eps \zeta'\right| \leq |x'-1|\nld{\eps}\nld{\zeta'} \leq C\nh{\eps}^2$, \item still thanks to this proposition, we have $\left| \int R(\eps)\zeta\right| \leq \ninf{\zeta}\nlu{R(\eps)}\leq C\nh{\eps}^2$. \end{enumerate} We have finally \begin{equation} \label{eq:Jprime} J' = \beta\lambda_0\int \eps\Qp +K(\eps)\end{equation} where $K(\eps)$ verifies $|K(\eps)|\leq C\nh{\eps}^2$. We now use identity \eqref{eq:Flin} which claims \[F(u(t))=F(u_0)=F(Q)+\frac{1}{2}(L\eps,\eps)+K'(\eps)\] with $|K'(\eps)|\leq C\nh{\eps}^3$. In other words, we have $(L\eps,\eps)+2K'(\eps)=2[F(u_0)-F(Q)] = 2[F(u_{0,n_0})-F(Q)] = -\gamma_{n_0}$ with $\gamma_{n_0}>0$, since $\nld{u_{0,n_0}}=\nld{Q}$ and $E(u_{0,n_0})<E(Q)$ by construction of $u_{0,n_0}$. To estimate the term $(L\eps,\eps)$, we use lemma \ref{th:coercif}, so that if we denote $a(t)=\int \eps\Qp$, we obtain \[a^2(t)\geq \frac{K_1}{K_2}\nh{\eps}^2 -\frac{1}{K_2}(L\eps,\eps) = \frac{\gamma_{n_0}}{K_2}+\frac{K_1}{K_2}\nh{\eps}^2 +\frac{2}{K_2}K'(\eps). \] Since $|K'(\eps)|\leq C\nh{\eps}^3$ and $\nh{\eps}\leq C\eps_0$, then if we take $\eps_0$ small enough, we have \[a^2(t)\geq K\nh{\eps}^2+\kappa_{n_0}\] with $K,\kappa_{n_0}>0$. In particular, $a^2(t)\geq \kappa_{n_0}>0$, thus $a$ keeps a constant sign, say positive. Then we have \[a(t)\geq \sqrt{K\nh{\eps}^2+\kappa_{n_0}} \geq \sqrt{\frac{K}{2}}\nh{\eps}+ \sqrt{\frac{\kappa_{n_0}}{2}} = K'\nh{\eps}+\kappa'_{n_0}. \] But from \eqref{eq:Jprime}, we also have $J'(t)=\beta\lambda_0a(t)+K(\eps)$ with $|K(\eps)|\leq C\nh{\eps}^2$, and so: \[J'(t) \geq \beta\lambda_0K'\nh{\eps}+\beta\lambda_0\kappa'_{n_0} -C\nh{\eps}^2 \geq \beta\lambda_0\kappa'_{n_0} = \theta_{n_0}>0\] if we choose as previously $\eps_0$ small enough. But it implies that $J(t)\geq \theta_{n_0}t+J(0) \vers +\infty$ as $t\to+\infty$, which contradicts the first step and concludes the proof of the theorem. Note that if $a(t)<0$, it is easy to show by the same arguments that $J'(t)\leq \theta'_{n_0}<0$, so $\lim_{t\to +\infty} J(t) = -\infty$ and then the same conclusion.

\section{Construction of a special solution by compactness} \label{sec:compactness}

In this section, we prove the existence of a special solution by a compactness method. This result is of course weaker than theorem \ref{th:main}, but it does not require the existence of $\Y_{\pm}$ proved in \cite{pegoweinstein}.

\subsection{Construction of the initial data}

Now theorem \ref{th:2} is proved, we can change $T_n$ obtained in \eqref{eq:th2} in the \emph{first} time which realizes this. In other words: \[ \exists\delta>0, \forall n\geq 1, \exists T_n\in\R_+ \m{ such that } \begin{cases} \inf_{y\in\R} \nh{u_n(T_n)-Q(\cdot-y)}=\delta\\ \forall t\in [0,T_n],\ \inf_{y\in\R} \nh{u_n(t)-Q(\cdot-y)}\leq\delta \end{cases}. \]

\begin{rem} \label{rem:lip}
We have $T_n\vers +\infty$. Indeed we would have $T_n<T_0$ for all $n$ otherwise (after passing to a subsequence). But by Lipschitz continuous dependence on the initial data (see \cite[corollary 2.18]{kpv}), we would have for $n$ large enough \[\sup_{t\in [0,T_0]} \nh{u_n(t)-Q(\cdot-t)} \leq K\nh{u_{0,n}-Q}. \] But since $\nh{u_{0,n}-Q}\tend 0$ by \eqref{eq:uon}, we would have $\inf_{y\in\R} \nh{u_n(t)-Q(\cdot-y)}\leq \frac{\delta}{2}$ for $n$ large enough and for all $t\in [0,T_0]$, which is wrong for $t=T_n\in [0,T_0]$.
\end{rem}

\bigskip

Now we can take $\delta$ smaller than $\eps_0$, so that $u_n(t)\in U_{\eps_0}$ for all $t\in [0,T_n]$ and so lemma \ref{th:decomp} applies: we can define $x_n(t)=\alpha(u_n(t))$ (notice that $x_n(0)=\alpha(u_{0,n})=0$) such that $\eps_n(t)=u_n(t,\cdot+x_n(t))-Q$ verifies \[\forall t\in [0,T_n], \left\{ \begin{array}{l} ( \eps_n(t),Q')=0\\ \nh{\eps_n(t)} \leq C\inf_{y\in\R} \nh{u_n(t)-Q(\cdot-y)} \leq C\delta \end{array} \right. .\] Moreover, for $t=T_n$, we have more precisely \begin{equation} \label{eq:epsn} \delta \leq \nh{\eps_n(T_n)}\leq C\delta. \end{equation} In particular, $\{\eps_n(T_n)\}$ is bounded in $H^1$, and so by passing to a subsequence, we can define \[\eps_n(T_n) \cvf \eps_{\infty} \m{ in } H^1 \m{ (weakly)} \m{ and } v_0=\eps_{\infty}+Q. \]

\begin{rem}
\begin{enumerate}
\item As announced in the introduction, one of the most important points in this section is to prove that we have constructed a non trivial object, \emph{i.e.} $v_0$ is not a soliton (proposition \ref{th:princ}). This fact is quite natural since $v_0$ is the weak limit of $u_n(T_n,\cdot+x_n(T_n))$ which contains a persisting defect $\eps_n(T_n)$.
\item Since the proof of proposition \ref{th:princ} is mainly based on evaluating $L^2$ norms, the following lemma will be useful.
\end{enumerate}
\end{rem}

\begin{lem} \label{th:epsgrand}
There exists $C_0>0$ such that, for $n$ large enough, $\nld{\eps_n(T_n)} \geq C_0\delta$.
\end{lem}

\begin{proof}
It comes from the conservation of the Weinstein's functional $F$ in time. In fact, we can write $F(Q+\eps_n(T_n))=F(Q+\eps_n(0))$ where $\eps_n(0)=u_{0,n}-Q$ verifies $\nh{\eps_n(0)}\tend 0$ by \eqref{eq:uon}. Then by \eqref{eq:Flin} \[F(Q)+\frac{1}{2}(L\eps_n(T_n),\eps_n(T_n)) +K(\eps_n(T_n)) = F(Q)+\frac{1}{2}(L\eps_n(0),\eps_n(0))+K(\eps_n(0))\] where $|K(a)|\leq C_1\nh{a}^3$. It comes \[\int \left[ \carre{(\partial_x\eps_n(T_n))} +\eps_n^2(T_n)-pQ^{p-1}\eps_n^2(T_n)\right] \leq C\nh{\eps_n(0)}^2 +K(\eps_n(0))-K(\eps_n(T_n))\] and so \[\nh{\eps_n(T_n)}^2 \leq C\int \eps_n^2(T_n)+C\nh{\eps_n(0)}^2 +C_1\nh{\eps_n(0)}^3+C_1\nh{\eps_n(T_n)}^3. \] Since $\nh{\eps_n(0)}\vers 0$, then by \eqref{eq:epsn} we have for $n$ large enough \[\nh{\eps_n(T_n)}^2 \leq C\int\eps_n^2(T_n) +C_1C\delta\nh{\eps_n(T_n)}^2+\frac{\delta^2}{4}. \] But if we choose $\delta$ small enough so that $C_1C\delta\leq\frac{1}{2}$, we obtain \[\frac{\delta^2}{2}\leq \frac{1}{2}\nh{\eps_n(T_n)}^2\leq C\int\eps_n^2(T_n) +\frac{\delta^2}{4}\] and finally $\int\eps_n^2(T_n)\geq \frac{\delta^2}{4C}$.
\end{proof}

\begin{prop} \label{th:princ}
For all $c>0$, $v_0\neq Q_c$.
\end{prop}

\begin{proof}
We proceed by contradiction: suppose that $v_n := u_n(T_n,\cdot+x_n(T_n))\cvf v_0=\eps_{\infty}+Q=Q_c$ weakly in $H^1$ for some $c>0$. We recall that it implies in particular that $v_n \vers Q_c$ strongly in $L^2$ on compacts as $n\to+\infty$.
\begin{itemize}
\item \emph{Decomposition of $v_n$:} Let $\varphi \in C^{\infty}(\R,\R)$ equals to $0$ on $(-\infty,-1]$ and $1$ on $[0,+\infty)$. Now let $A\gg 1$ to fix later and define $\varphi_A(x)=\varphi(x+A)$, so that $\varphi_A(x)=0$ if $x\leq -A-1$ and $1$ if $x\geq -A$. We also define $h_n=(1-\varphi_A)v_n$, $Q_c^A=Q_c\varphi_A$ and $z_n=\varphi_Av_n-\varphi_AQ_c = \varphi_A(v_n-Q_c)$, so that \[v_n = (1-\varphi_A)v_n+\varphi_Av_n = h_n+z_n+Q_c^A. \]
\item \emph{Estimation of $\nld{z_n}$:} \begin{align*} \int z_n^2 &= \int \carre{(v_n-Q_c)}\varphi_A^2 \leq \int_{-A-1}^{A+1} \carre{(v_n-Q_c)} + \int_{x>A+1} \carre{(v_n-Q_c)}\\ &\leq \int_{-A-1}^{A+1} \carre{(v_n-Q_c)} + 2\int_{x>A+1} v_n^2 +2\int_{x>A+1} Q_c^2 = I+J+K. \end{align*} Notice that $I\tend 0$ since $v_n\tend Q_c$ in $L^2$ on compacts. Moreover, thanks to exponential decay of $Q_c$, we have $K\leq Ce^{-2\sqrt c A}$. Finally, we have $J\leq Ce^{-A/4}$ with $C$ independent of $n$ by remark \ref{th:monotonie}. In summary, there exists $\rho>0$ such that $\int z_n^2 \leq Ce^{-\rho A}$ if $n\geq n(A)$.
\item \emph{Mass balance:} On one hand, we have by \eqref{eq:uon} and mass conservation $\int v_n^2 = \int u_{0,n}^2 = \int Q^2$. On the other hand, we can calculate \[\int v_n^2 = \int h_n^2 + \int \carre{(Q_c^A+z_n)} + 2\int_{-A-1}^{-A} v_n^2\varphi_A(1-\varphi_A). \] But since $v_n\vers Q_c$ on compacts, we have $2\int_{-A-1}^{-A} v_n^2\varphi_A(1-\varphi_A) \tend 2\int_{-A-1}^{-A} Q_c^2\varphi_A(1-\varphi_A) \leq Ce^{-\rho A}$. Consequently, \[\int Q^2 = \int h_n^2+\int \Carre{Q_c^A} +2\int Q_c^A z_n+\int z_n^2 +a_n^A\] where $a_n^A\geq 0$ verifies $a_n^A\leq Ce^{-\rho A}$ for $n\geq n(A)$. Thanks to the previous estimation of $\nld{z_n}$ and the Cauchy-Schwarz inequality, we deduce that \[\int Q^2 = \int h_n^2 +\int \Carre{Q_c^A} +a'^A_n\] where $a'^A_n$ verifies $|a'^A_n|\leq Ce^{-\rho A}$ for $n\geq n(A)$. But \[ \int \Carre{Q_c^A} = \int Q_c^2\varphi_A^2 = \int Q_c^2 +\int Q_c^2(\varphi_A^2-1) \leq \int Q_c^2 +\int_{x<-A} Q_c^2 \leq \int Q_c^2 + Ce^{-\rho A} \] and $\int Q_c^2 = c^{-\beta}\int Q^2$ with $\beta>0$ since $p>5$ (see claim \ref{th:solitons}). In conclusion, we have the mass balance \begin{equation} \label{eq:bilanmasse} (1-c^{-\beta})\nld{Q}^2 = \nld{h_n}^2 + a''^A_n \end{equation} where $a''^A_n$ still verifies $|a''^A_n|\leq Ce^{-\rho A}$ for $n\geq n(A)$.
\item \emph{Upper bound of $\nld{h_n}$:} We remark that for $n\geq n(A)$, $\nld{h_n}\leq C_1\delta$. Indeed, thanks to \eqref{eq:epsn}, we have \[\nld{h_n} \leq \nld{(1-\varphi_A)Q}+\nld{\eps_n(T_n)} \leq Ce^{-\rho A} +C\delta \leq C_1\delta\] if we definitively fix $A$ large enough so that $e^{-\rho A}\leq\delta^3$ (the power $3$ will be useful later in the proof).
\item \emph{Upper bound of $|c-1|$:} Thanks to the previous point and mass balance \eqref{eq:bilanmasse}, we have $|1-c^{-\beta}|\leq C\delta^2$. We deduce that $c$ is close to $1$, and so by Taylor's theorem that $|c-1|\leq K|1-c^{-\beta}|\leq C\delta^2$.
\item \emph{Lower bound of $\nld{h_n}$:} We now prove that for $n\geq n(A)$, $\nld{h_n}\geq C_2\delta$. Firstly, we have by lemma \ref{th:epsgrand}: \begin{align*} C_0\delta &\leq \nld{\eps_n(T_n)} = \nld{v_n-Q} = \nld{h_n+Q_c^A+z_n-Q}\\ &\leq \nld{h_n} +\nld{z_n}+\nld{Q_c^A-Q_c}+\nld{Q_c-Q} = \nld{h_n}+\nld{Q_c-Q}+b_n^A \end{align*} where $b_n^A= \nld{z_n}+\nld{Q_c^A-Q_c} \geq 0$ verifies $b_n^A \leq Ce^{-\rho A}$ for $n\geq n(A)$. Moreover, if we denote $f(c)=\nld{Q_c-Q}^2$ for $c>0$, then $f$ is $C^{\infty}$ and $f(c)\geq 0=f(1)$, hence $1$ is a minimum of $f$, $f'(1)=0$ and so by Taylor's theorem: $f(c)\leq C\carre{(c-1)}$, \emph{i.e.} $\nld{Q_c-Q}\leq C|c-1|$. Thanks to the previous point, we deduce that \[C_0\delta\leq \nld{h_n}+K\delta^2+b_n^A \leq \nld{h_n}+C\delta^2. \] Finally, if we choose $\delta$ small enough so that $C\delta\leq \frac{C_0}{2}$, we reach the desired inequality.
\item \emph{Energy balance:} We now use the conservation of Weinstein's functional and \eqref{eq:Flin} to write \[ F(u_0) =F(v_n)=F(Q+\eps_n(T_n)) = F(Q)+\frac{1}{2}(L\eps_n(T_n),\eps_n(T_n))+K(\eps_n(T_n))\] where $|K(\eps_n(T_n))|\leq C\nh{\eps_n(T_n)}^3 \leq C\delta^3$ by \eqref{eq:epsn}. Now we decompose $\eps_n(T_n)$ in \[\eps_n(T_n) = v_n-Q = h_n+z_n+Q_c^A-Q = (Q_c-Q)+(Q_c^A-Q_c)+(z_n+h_n)\] in order to expand \begin{align*} (L\eps_n(T_n),\eps_n(T_n)) &= (L(Q_c-Q),Q_c-Q)+(L(z_n+h_n),z_n+h_n)\\ & \quad +(L(Q_c^A-Q_c),Q_c^A-Q_c) +2(L(Q_c-Q),z_n+h_n)\\ &\quad +2(L(Q_c-Q),Q_c^A-Q_c)+2(L(Q_c^A-Q_c),z_n+h_n). \end{align*} We recall that $(La,b)=-\int a''b+\int ab-p\int Q^{p-1}ab$, and so by the Cauchy-Schwarz inequality: $|(La,b)|\leq (\nld{a''}+C\nld{a})\nld{b}$. Since we have $\nld{z_n+h_n}\leq \nld{z_n}+\nld{h_n} \leq Ce^{-\rho A} +C_1\delta\leq C\delta$, we can estimate \[ |(L(Q_c-Q),z_n+h_n)| \leq (\nld{Q''_c-Q''}+C\nld{Q_c-Q})\nld{z_n+h_n} \leq C|c-1|\cdot C\delta \leq C\delta^3. \] Similarly, we have \begin{align*} |(L(Q_c^A-Q_c),z_n+h_n)| &\leq (\nld{\varphi''_AQ_c} +2\nld{\varphi'_AQ'_c} + \nld{(\varphi_A-1)Q''_c}\\ & \qquad +C\nld{Q_c^A-Q_c})\nld{z_n+h_n}\\ &\leq Ce^{-\rho A}\cdot C\delta \leq C\delta^3. \end{align*} Moreover, we have by integrating by parts $(La,b)=\int a'b' +\int ab-p\int Q^{p-1}ab$, and so $|(La,b)|\leq C\nh{a}\nh{b}$. It implies that \[\left\{ \begin{aligned} |(L(Q_c-Q),Q_c-Q)| &\leq C\nh{Q_c-Q}^2 \leq C\carre{(c-1)} \leq C\delta^3, \\ |(L(Q_c^A-Q_c),Q_c^A-Q_c)| &\leq C\nh{Q_c^A-Q_c}^2\leq Ce^{-2\rho A} \leq C\delta^3, \\ |(L(Q_c-Q),Q_c^A-Q_c)| &\leq C\nh{Q_c-Q}\nh{Q_c^A-Q_c} \leq C|c-1|\cdot Ce^{-\rho A} \leq C\delta^3, \end{aligned} \right.\] thanks to the estimate on $|c-1|$ previously found. For the last term, we have \[(L(h_n+z_n),h_n+z_n) = \nh{h_n+z_n}^2 -p\int Q^{p-1}\carre{(h_n+z_n)}\] and \begin{align*} \int Q^{p-1}\carre{(h_n+z_n)} &\leq 2\int Q^{p-1}h_n^2 +2\int Q^{p-1}z_n^2 \leq 2\int \carre{(1-\varphi_A)} Q^{p-1} v_n^2 +2\ninf{Q}^{p-1}\int z_n^2\\ &\leq 2\int_{x<-A} Q^{p-1}v_n^2 +2\ninf{Q}^{p-1}\int z_n^2. \end{align*} But $\ninf{v_n}\leq C\nh{v_n}\leq C(\nh{\eps_n(T_n)}+\nh{Q}) \leq C(K\delta+\nh{Q})=K'$, and so $\int_{x<-A} Q^{p-1}v_n^2 \leq C\int_{x<-A} Q^{p-1} \leq Ce^{-\rho A}$. As $\int z_n^2 \leq Ce^{-\rho A}$, we have \[F(u_0)=F(Q)+\frac{1}{2}\nh{h_n+z_n}^2 +d_n^A \geq F(Q)+\frac{1}{2}\nld{h_n+z_n}^2 +d_n^A\] where $|d_n^A|\leq C\delta^3$ for $n\geq n(A)$. Moreover we have \[ \nld{h_n+z_n}^2 -\nld{h_n}^2 \leq \nld{z_n}^2 +2\nld{z_n}\nld{h_n} \leq Ce^{-2\rho A}+2Ce^{-\rho A}\cdot C_1\delta \leq C\delta^3. \] Finally, energy balance provides us, for some $N$ large enough, \[F(u_0) \geq F(Q) +\frac{1}{2}\nld{h_N}^2 +d'\] with $|d'|\leq C\delta^3$.
\item \emph{Conclusion:} Since $F(u_0)<F(Q)$ by hypothesis, we obtain $\nld{h_N}^2 \leq C\delta^3$. But we also have by the lower bound of $\nld{h_n}$: $\nld{h_N}^2 \geq C_2^2\delta^2$. Gathering both information, we obtain $\frac{C_2^2}{C}\leq \delta$, which is clearly a contradiction if we choose $\delta$ small enough, and so concludes the proof of proposition \ref{th:princ}. \qedhere
\end{itemize}
\end{proof}

\subsection{Weak continuity of the flow}

The main idea to obtain the special solution is to reverse the weak convergence of $v_n$ to $v_0$ in time and in space, using the fact that $u(t,x)$ is a solution of \eqref{eq:gKdV} if and only if $u(-t,-x)$ is also a solution. More precisely, we define $w_0 = \check{v}_0\in H^1(\R)$, \emph{i.e.} for all $x\in\R$, $w_0(x)=v_0(-x)$.

\begin{rem} \label{th:remarque}
For all $c>0$ and all $x_0\in\R$, we have \[w_0\neq Q_c(\cdot+x_0). \] In fact, otherwise and since $Q_c$ is even, we would have $v_0(x)=Q_c(x-x_0)$. But $v_n-Q=\eps_n(T_n)$ and $(\eps_n(T_n),Q')= (v_n,Q')=0$, so by weak convergence in $H^1$, $(v_0,Q')=0$. Thus we would have $\int Q_c(x-x_0)Q'(x)\,dx=0$, and if we show that $x_0=0$, we shall reach the desired contradiction since we have $v_0\neq Q_c$ for all $c>0$ by proposition \ref{th:princ}. To show this, consider $f(a)=\int Q_c(x-a)Q'(x)\,dx$ for $a\in\R$, which is odd since $Q_c$ is even and $Q'$ odd. In particular, $f(0)=0$, and it is enough to show that $f(a)<0$ for $a>0$ to conclude (because we shall have $f(a)>0$ for $a<0$ by parity). But using again the parity of $Q_c$ and $Q'$, we have \[ f(a) = \int_0^a [Q_c(a-x)-Q_c(a+x)]Q'(x)\,dx + \int_a^{+\infty} [Q_c(x-a)-Q_c(x+a)]Q'(x)\,dx. \] Since $Q'$ is negative and $Q_c$ is strictly decreasing on $\R_+$, both integrals are negative, and so $f(a)<0$ for $a>0$, as we desired.
\end{rem}

\begin{rem}
\begin{enumerate}
\item Now, $w_0$ being constructed, we show that the associated solution $w(t)$ is defined for all $t$ positive, and can be seen as a weak limit (proposition \ref{th:cvf}) in order to prove the convergence of $w(t)$ to a soliton.
\item The main ingredient of the proof of proposition \ref{th:cvf} is the following lemma of weak continuity of the flow, whose proof is inspired by \cite[theorem 5]{martel:bo}. This proof is long and technical, and thus is not completely written in this paper.
\end{enumerate}
\end{rem}

\begin{lem} \label{th:wcf}
Suppose that $z_{0,n}\cvf z_0$ in $H^1$, and that there exist $T>0$ and $K>0$ such that the solution $z_n(t)$ corresponding to initial data $z_{0,n}$ exists for $t\in [0,T]$ and $\supt \nh{z_n(t)} \leq K$. Then for all $t\in [0,T]$, the solution $z(t)$ such that $z(0)=z_0$ exists, and $z_n(T)\cvf z(T)$ in $H^1$.
\end{lem}

\begin{proof}[Sketch of the proof]
Let $T^* = T^*(\nht{z_0})>0$ be the maximum time of existence of the solution $z(t)$ associated to $z_0$, well defined by \cite[corollary 2.18]{kpv} since $s=\frac{3}{4} > \frac{p-5}{2(p-1)} = s_c(p)$. We distinguish two cases, whether $T<T^*$ or not, and we show that this last case is in fact impossible.
\begin{description}
\item[1st case:] Suppose that $T<T^*$. As $z(t)$ exists for $t\in [0,T]$ by hypothesis, it is enough to show that $z_n(T)\cvf z(T)$ in $H^1$. But since $C_0^{\infty}$ is dense in $H^{-1}$ and $\nh{z_n(T)-z(T)}\leq \nh{z_n(T)}+\nh{z(T)} \leq K'$, it is enough to show that $z_n(T)\vers z(T)$ in $\mathcal{D}'(\R)$. It is the end of this case, very similar to the proof in \cite{martel:bo} (but using a $H^3$ regularization and so using some arguments like in \cite[section 3.4]{martel:Nsoliton}), which is technical and not written in this paper consequently.
\item[2nd case:] Suppose that $T^*\leq T$ and let us show that it implies a contradiction. Indeed, there would exist $T'<T^*$ such that $\nht{z(T')} \geq 2K$ (where $K$ is the same constant as in the hypothesis of the lemma). But we can apply the first case with $T'$ instead of $T$, so that $z_n(T')\cvf z(T')$ in $H^1$, and since $\nh{z_n(T')}\leq K$, we obtain by weak convergence $\nht{z(T')}\leq \nh{z(T')} \leq K$, and so the desired contradiction. \qedhere
\end{description}
\end{proof}

\begin{prop} \label{th:cvf}
The solution $w(t)$ of \eqref{eq:gKdV} such that $w(0)=w_0$ is defined for all $t\geq 0$, and $u_n(T_n-t,x_n(T_n)-\cdot) \cvf w(t)$ in $H^1$.
\end{prop}

\begin{proof}
As the assumption is clear for $t=0$, we fix $T>0$ and we show it for this $T$. Since $\lim_{n\to+\infty} T_n = +\infty$ by remark \ref{rem:lip}, then for $n\geq n_0$, we have $T_n\geq T$. As a consequence, for $n\geq n_0$ and for $t\in [0,T]$, $z_n(t)=u_n(T_n-t,x_n(T_n)-\cdot)$ is well defined, solves \eqref{eq:gKdV}, and has for initial data \[z_n(0) = u_n(T_n,x_n(T_n)-\cdot) = \check{v}_n \cvf \check{v}_0 = w_0\quad \m{in } H^1. \] Moreover, we have \begin{align*} \nh{z_n(t)} &= \nh{u_n(T_n-t,x_n(T_n)-\cdot)}\\ &\leq \nh{\eps_n(T_n-t,x_n(T_n)-x_n(T_n-t)-\cdot)} +\nh{Q(x_n(T_n)-x_n(T_n-t)-\cdot)}\\ &\leq \nh{\eps_n(T_n-t)} +\nh{Q} \leq C\delta+\nh{Q} = K. \end{align*} By lemma \ref{th:wcf}, we deduce that $w$ exists on $[0,T]$, and $z_n(T)\cvf w(T)$ in $H^1$.
\end{proof}

\subsection{Exponential decay on the left of $w$}

The goal of this section is to prove an exponential decay on the "left" of $w$, using the exponential decay of $u_n$ on the right. Indeed, since $\eps_n(T_n-t)=u_n(T_n-t,\cdot+x_n(T_n-t))-Q$ verifies $(\eps_n(T_n-t),Q')=0$ and $\nh{\eps_n(T_n-t)}\leq C\delta$ for all $t\in [0,T_n]$, then $u_n(T_n-t)$ is in the same situation as the situation of $u$ summed up just before proposition \ref{th:eqeps}, with $\delta$ instead of $\eps_0$ for the small parameter. In particular, by remark \ref{th:monotonie}, inequality \eqref{eq:expdecay} holds for $u_n(T_n-t)$ with $C$ independent of $n$ if we choose $\delta$ small enough. In other words, we have for all $t\geq 0$ and $x_0>0$ (and $n$ large enough): \begin{equation} \label{eq:expdecayun} \int_{x>x_0} (u_{nx}^2+u_n^2)(T_n-t,x+x_n(T_n-t))\,dx \leq Ce^{-x_0/4}. \end{equation} But before passing to the limit, we have to define the "left" of $w$, \emph{i.e.} the center of mass $x_w(t)$ of $w(t)$.

\begin{lem} \label{th:wprocheQ}
There exists $C>0$ such that, for all $t\geq 0$, $\inf_{y\in\R} \nh{w(t)-Q(\cdot-y)}\leq C\delta$.
\end{lem}

\begin{proof}
Fix $t\geq 0$ and $n_0\geq 0$ such that for $n\geq n_0$, $T_n\geq t$. Since $Q$ is even, we have \[ \eps_n(T_n-t,x_n(T_n)-x_n(T_n-t)-\cdot) = u_n(T_n-t,x_n(T_n)-\cdot) -Q(\cdot-x_n(T_n)+x_n(T_n-t)). \] Now if we denote $w_n(t)=u_n(T_n-t,x_n(T_n)-\cdot)$ and $y_n(t)=x_n(T_n)-x_n(T_n-t)$, we have \[\nh{w_n(t)-Q(\cdot-y_n(t))} = \nh{\eps_n(T_n-t)} \leq C\delta. \] But following the remark done at the beginning of this section, proposition \ref{th:eqeps} is still valid, and so $|x'_n(t)-1|\leq C\delta$ for $t\in [0,T_n]$. We deduce that $y_n(t)= \int_{T_n-t}^{T_n} x'_n(s)\,ds = \int_{T_n-t}^{T_n} (x'_n(s)-1)\,ds +t$ verifies $|y_n(t)|\leq C\delta t+t = Ct$. By passing to a subsequence, we can suppose that $\lim_{n\to\infty} y_n(t) = y(t)$. But now we can write \[ \nh{w_n(t)-Q(\cdot-y(t))} \leq C\delta +\nh{Q-Q(\cdot+(y_n(t)-y(t)))} \leq C'\delta \] for $n\geq N(t,\delta)$ by lemma \ref{th:Qmass}. Finally, since $w_n(t)\cvf w(t)$ in $H^1$ by proposition \ref{th:cvf}, we obtain by weak convergence $\nh{w(t)-Q(\cdot-y(t))} \leq C'\delta$, and the result follows.
\end{proof}

We can now choose $\delta$ small enough so that $C\delta\leq\eps_0$, and so we can define $x_w(t)=\alpha(w(t))$ by lemma \ref{th:decomp}, with notably $\nh{w(t,\cdot+x_w(t))-Q}\leq C\delta$. But to exploit \eqref{eq:expdecayun}, we have to show first that $y_n(t)=x_n(T_n)-x_n(T_n-t)$ is close to $x_w(t)$ for all $t$.

\begin{lem} \label{th:masscenters}
There exists $C>0$ such that: $\forall t\geq 0, \exists n_0\geq 0, \forall n\geq n_0,\ |x_w(t)-y_n(t)|\leq C\delta$.
\end{lem}

\begin{proof}
Let $t\geq 0$ and $n$ large enough such that $T_n\geq t$. We keep notation $w_n(t)$ and $y_n(t)$ of the previous proof , where we have already remarked that $|y_n(t)|\leq Ct$. For the same reason, we have $|x_w(t)-y_n(t)|\leq \Omega t$. Now choose $A(t)\gg 1$ such that ${\|Q\|}_{L^2(|x|\geq A(t)-\Omega t)}\leq\delta$. Since $w_n(t)\cvf w(t)$ in $H^1$, then for $n\geq n_0$, we have ${\| w_n(t)-w(t)\|}_{L^2(|x|\leq A(t))} \leq\delta$. Moreover, \[\nh{w(t)-Q(\cdot-x_w(t))}\leq C\delta \quad\m{and}\quad \nh{w_n(t)-Q(\cdot-y_n(t))}\leq C\delta, \] and so by the triangle inequality: ${\|Q(\cdot-x_w(t))-Q(\cdot-y_n(t))\|}_{L^2(|x|\leq A(t))} \leq C\delta$. We deduce that for $n\geq n_0$: \begin{align*} \nld{Q-Q(\cdot+x_w(t)-y_n(t))} &\leq \sqrt2 {\|Q-Q(\cdot+x_w(t)-y_n(t))\|}_{L^2(|x|\leq A(t))}\\ &\qquad + \sqrt2 {\|Q-Q(\cdot+x_w(t)-y_n(t))\|}_{L^2(|x|\geq A(t))}\\ &\leq C\delta +\sqrt2 {\|Q\|}_{L^2(|x|\geq A(t))} +\sqrt2 {\|Q(\cdot+x_w(t)-y_n(t))\|}_{L^2(|x|\geq A(t))}\\ &\leq C\delta +2\sqrt2 {\|Q\|}_{L^2(|x|\geq A(t)-\Omega t)} \leq C\delta. \end{align*} We conclude by choosing $\delta$ small enough so that $C\delta\leq A_0$, where $A_0$ is defined in lemma \ref{th:Qmass}, and we apply this lemma to reach the desired inequality (note that the lemma holds of course with the $L^2$ norm instead of the $H^1$ one).
\end{proof}

If we choose $\delta$ small enough so that $C\delta\leq 1$ (for example) in lemma \ref{th:masscenters}, we can now prove the following proposition.

\begin{prop} \label{th:expdecayw}
There exists $C>0$ such that, for all $t\geq 0$ and all $x_0>0$, \[\int_{x<-x_0-1} (w_x^2+w^2)(t,x+x_w(t))\,dx \leq Ce^{-x_0/4}. \]
\end{prop}

\begin{proof}
Let $t\geq 0$, $x_0>0$ and $n\geq n_0$ where $n_0$ is defined in lemma \ref{th:masscenters}. From \eqref{eq:expdecayun} and the substitution $y=x_n(T_n)-x_n(T_n-t)-x=y_n(t)-x$, we obtain \[\int_{x<y_n(t)-x_0} (u_{nx}^2+u_n)(T_n-t,x_n(T_n)-x)\,dx \leq Ce^{-x_0/4}. \] If we still denote $w_n(t)=u_n(T_n-t,x_n(T_n)-\cdot)$, we deduce by lemma \ref{th:masscenters} that \[\int_{x<-x_0-1+x_w(t)} (w_{nx}^2+w_n^2)(t,x)\,dx \leq Ce^{-x_0/4}. \] But $w_n(t)\cvf w(t)$ in $H^1$, so $w_n(t)\cvf w(t)$ and $w_{nx}(t)\cvf w_x(t)$ in $L^2$. Moreover, since $\psi = \mathbbm{1}_{(-\infty,-x_0-1+x_w(t))} \in L^{\infty}$, then $w_n(t)\psi\cvf w(t)\psi$ and $w_{nx}(t)\psi\cvf w_x(t)\psi$ in $L^2$, thus by weak convergence $\int_{x<-x_0-1+x_w(t)} w^2(t,x)\,dx \leq Ce^{-x_0/4}$ and the same inequality for $w_x$, so the result follows by sum.
\end{proof}

\subsection{Asymptotic stability and conclusion}

The final ingredient to prove that $w(t)$ is a special solution is the theorem of asymptotic stability proved by Martel and Merle in \cite{martel:general}. Indeed, thanks to lemma \ref{th:wprocheQ}, we can apply this theorem with $c_0=1$ if we choose $\delta$ small enough such that $C\delta<\alpha_0$. We obtain $c_+>0$ and $t\mapsto \rho(t)\in\R$ such that \begin{equation} \label{eq:as} {\|w(t)-Q_{c_+}(\cdot-\rho(t))\|}_{H^1(x>t/10)} \xrightarrow[t\to +\infty]{} 0. \end{equation}

\begin{rem}
As usual, $\rho(t)$ and $c_+$ are defined in \cite{martel:general} by a lemma of modulation close to $Q$, which gives the estimations: $\nh{w(t)-Q_{c_+}(\cdot-\rho(t))}\leq C\delta$, $|\rho'(t)-1|\leq C\delta$ and $|c_+-1|\leq C\delta$. We deduce that \begin{align*} \nh{Q-Q(\cdot+\rho(t)-x_w(t))} &= \nh{Q(\cdot-\rho(t))-Q(\cdot-x_w(t))}\\ &\leq \nh{Q-Q_{c_+}} +\nh{w(t)-Q_{c_+}(\cdot-\rho(t))}\\ &\qquad + \nh{w(t)-Q(\cdot-x_w(t))}\\ &\leq K|c_+-1| +C\delta+C'\delta \leq C''\delta. \end{align*} Now if we choose $\delta$ small enough, then $C''\delta\leq A_0$ and lemma \ref{th:Qmass} gives $|x_w(t)-\rho(t)|\leq C\delta \leq 1$. Finally, proposition \ref{th:expdecayw} becomes \begin{equation} \label{eq:expdecayw} \forall t\geq 0,\forall x_0>2,\quad \int_{x<-x_0} (w_x^2+w^2)(t,x+\rho(t))\,dx \leq C'e^{-x_0/4}. \end{equation}
\end{rem}

\bigskip

We are now able to prove the main result of this section.

\begin{theo}[Existence of one special solution] \label{th:thprinc}
There exist $w(t)$ solution of \eqref{eq:gKdV} defined for all $t\geq 0$, $c_+>0$ and $t\mapsto \rho(t)$ such that:
\begin{enumerate}[(i)]
\item ${\|w(t)-Q_{c_+}(\cdot-\rho(t))\|}_{H^1(\R)} \xrightarrow[t\to +\infty]{} 0$,
\item $\forall c>0, \forall x_0\in\R, w(0)\neq Q_{c}(\cdot+x_0)$.
\end{enumerate}
\end{theo}

\begin{proof}
By remark \ref{th:remarque}, it is enough to prove (i). We have by the triangle inequality \begin{align*} {\|w(t)-Q_{c_+}(\cdot-\rho(t))\|}_{H^1(\R)}^2 &\leq {\|w(t)-Q_{c_+}(\cdot-\rho(t))\|}_{H^1(x>t/10)}^2 +2{\|w(t)\|}_{H^1(x<t/10)}^2\\ &\qquad + 2{\|Q_{c_+}(\cdot-\rho(t))\|}_{H^1(x<t/10)}^2 = \mathbf{I+II+III}. \end{align*} Since $|\rho'(t)-1|\leq C\delta\leq \frac{1}{10}$ if we choose $\delta$ small enough, then $|\rho(t)-t-\rho(0)|\leq \frac{1}{10}t$, and so if we denote $\rho_0=\rho(0)\in\R$, we have $\frac{t}{10}-\rho(t)\leq -\frac{4}{5}t-\rho_0$. We can now estimate:
\begin{itemize}
\item $\mathbf{I} \xrightarrow[t\to +\infty]{} 0$ by \eqref{eq:as}.
\item For $t$ large enough, we have $\frac{4t}{5}+\rho_0>2$, and so \eqref{eq:expdecayw} gives \begin{align*} \frac{1}{2}\mathbf{II} &= \int_{x<t/10} (w_x^2+w^2)(t,x)\,dx = \int_{x<t/10-\rho(t)} (w_x^2+w^2)(t,x+\rho(t))\,dx\\ &\leq \int_{x<-4t/5-\rho_0} (w_x^2+w^2)(t,x+\rho(t))\,dx \leq Ce^{-t/5} \xrightarrow[t\to +\infty]{} 0. \end{align*}
\item Finally, since $(Q'^2_{c_+}+Q_{c_+}^2)(x)\leq Ce^{2\sqrt{c_+}x}$ for all $x\in\R$ (see claim \ref{th:solitons}), we have \begin{align*} \frac{1}{2}\mathbf{III} &= \int_{x<t/10} (Q'^2_{c_+}+Q_{c_+}^2)(x-\rho(t))\,dx = \int_{x<t/10-\rho(t)} (Q'^2_{c_+}+Q_{c_+}^2)(x)\,dx\\ &\leq \int_{x<-4t/5-\rho_0} (Q'^2_{c_+}+Q_{c_+}^2)(x)\,dx \leq C\int_{x<-4t/5-\rho_0} e^{2\sqrt{c_+}x}\,dx \leq Ce^{-\frac{8t}{5}\sqrt{c_+}} \xrightarrow[t\to +\infty]{} 0 \end{align*} which achieves the proof of theorem \ref{th:thprinc}. \qedhere
\end{itemize}
\end{proof}

\begin{cor} \label{th:scaling}
For all $c>0$, there exist $w_c(t)$ solution of \eqref{eq:gKdV} defined for all $t\geq 0$ and $t\mapsto \rho_c(t)$ such that: \begin{enumerate}[(i)] \item ${\|w_c(t,\cdot+\rho_c(t))-Q_c\|}_{H^1(\R)} \xrightarrow[t\to +\infty]{} 0$, \item $\forall c'>0, \forall x_0\in\R, w_c(0,\cdot+\rho_c(0))\neq Q_{c'}(\cdot+x_0)$. \end{enumerate}
\end{cor}

\begin{proof}
It is based on the scaling invariance of the \eqref{eq:gKdV} equation: if $u(t,x)$ is a solution, then for all $\lambda>0$, $\lambda^{\frac{2}{p-1}}u(\lambda^3t,\lambda x)$ is also a solution. For $c>0$ given, we thus define $w_c$ by $w_c(t)=\lambda_c^{\frac{2}{p-1}}w(\lambda_c^3t,\lambda_c x)$ with $\lambda_c=\sqrt\frac{c}{c_+}$, where $w$ and $c_+$ are defined above. Since $Q_c(x)=\lambda_c^{\frac{2}{p-1}}Q_{c_+}(\lambda_c x)$, then we have by substitution \begin{multline*} \nh{w(t)-Q_{c_+}(\cdot-\rho(t))}^2 = \lambda_c^{\frac{p-5}{p-1}} \Big( \nld{w_c(t/\lambda_c^3,\cdot+\rho(t)/\lambda_c)-Q_c}^2\\ +\frac{1}{\lambda_c^2} \nld{\partial_x[w_c(t/\lambda_c^3,\cdot+\rho(t)/\lambda_c)-Q_c]}^2 \Big). \end{multline*} We deduce that \[ \nh{w(t)-Q_{c_+}(\cdot-\rho(t))}^2 \geq \left\{ \begin{aligned} \lambda_c^{\frac{p-5}{p-1}} \nh{w_c(t/\lambda_c^3,\cdot+\rho(t)/\lambda_c)-Q_c}^2 &\quad \m{if } \lambda_c\leq 1\\ \lambda_c^{-\frac{p+3}{p-1}} \nh{w_c(t/\lambda_c^3,\cdot+\rho(t)/\lambda_c)-Q_c}^2 &\quad \m{if } \lambda_c\geq 1 \end{aligned} \right. , \] and so $\lim_{t\to+\infty} \nh{w_c(t/\lambda_c^3,\cdot+\rho(t)/\lambda_c)-Q_c} = 0$ in both cases by theorem \ref{th:thprinc}. We finally obtain (i) if we take $\rho_c(t)=\frac{\rho(\lambda_c^3 t)}{\lambda_c}$. For (ii), if we suppose that there exist $c'>0$ and $x_0\in\R$ such that $w_c(0,\cdot +\rho_c(0))=Q_{c'}(\cdot+x_0)$, then we get \[w_0 = Q_{\frac{c'c_+}{c}} \left(\cdot+\left( \sqrt\frac{c}{c_+}x_0-\rho_0\right) \right)\] which is a contradiction with remark \ref{th:remarque}.
\end{proof}

\section{Construction and uniqueness of a family of special solutions via the contraction principle} \label{sec:contraction}

In this section, we prove theorem \ref{th:main}. The proof is an extension to \eqref{eq:gKdV} of the method by fixed point developed in \cite{duymerle,duy}. To adapt the method to \eqref{eq:gKdV}, we use first information on the spectrum of the linearized operator around $Q(\cdot-t)$ due to \cite{pegoweinstein} (see proposition \ref{th:spectrum} in the present paper). Secondly, we rely on the Cauchy theory for \eqref{eq:gKdV} developed in \cite{kpv,kpvcritical}. Indeed, one of the main difficulties is the lack of a derivative due to the equation, but compensated by a smoothing effect already used in \cite{kpv,kpvcritical}.

\subsection{Preliminary estimates for the Cauchy problem}

Theorem 3.5 of \cite{kpv} and proposition 2.3 of \cite{kpvcritical} are summed up and adapted to our situation in proposition \ref{th:kpv} below. We note $W(t)$ the semigroup associated to the linear equation $\partial_t u+\partial_x^3u=0$.

\begin{notation}
Let $I\subset\R$ be an interval, $1\leq p,q\leq \infty$ and $g : \R\times I\to \R$. Then define \[ {\| g\|}_{L_x^pL_I^q} = \Puiss{\int_{-\infty}^{+\infty} \Puiss{\int_I {|g(x,t)|}^q\,dt}{p/q} dx}{1/p} \ ,\ {\| g\|}_{L_I^qL_x^p} = \Puiss{\int_I \Puiss{\int_{-\infty}^{+\infty} {|g(x,t)|}^p\,dx}{q/p} dt}{1/q} \] and $L_x^pL_I^q = \{ g\ |\ {\| g\|}_{L_x^pL_I^q}<+\infty\}$ and $L_I^qL_x^p = \{ g\ |\ {\| g\|}_{L_I^qL_x^p}<+\infty \}$. Finally, denote $L_x^pL_t^q = L_x^pL_{\R}^q$ and $L_t^qL_x^p = L_{\R}^qL_x^p$.
\end{notation}

\begin{prop} \label{th:kpv}
There exists $C>0$ such that for all $g\in L_x^1L_t^2$ and all $T\in\R$, \begin{align} {\left\| \drondx \int_t^{+\infty} W(t-t')g(x,t')\,dt' \right\|}_{L_{[T,+\infty)}^{\infty}L_x^2} &\leq C\nluld{g}, \label{eq:kpv1} \\ \Nlcld{\drondx \int_t^{+\infty} W(t-t')g(x,t')\,dt'} &\leq C\nluld{g} \label{eq:kpv2}. \end{align}
\end{prop}

\begin{proof}
\begin{enumerate}[(i)]
\item Inequality \eqref{eq:kpv1} comes from the dual inequality of (3.6) in \cite{kpv}, \emph{i.e.} \[ {\left\| \drondx \int_{-\infty}^{+\infty} W(-t')g(x,t')\,dt' \right\|}_{L_x^2} \leq C{\|g\|}_{L_x^1L_t^2}. \] Let $t\geq T$, we get for $\tilde{g}(x,t') = \mathbbm{1}_{[t,+\infty)}(t')g(x,t')$: \[ {\left\| \drondx \int_t^{+\infty} W(-t')g(x,t')\,dt' \right\|}_{L_x^2} = {\left\| \drondx \int_{-\infty}^{+\infty} W(-t')\tilde{g}(x,t')\,dt' \right\|}_{L_x^2} \leq C\nluld{g} \] and so the desired inequality since $W$ is unitary on $L^2$.
\item Inequality \eqref{eq:kpv2} comes from inequalities (2.6) and (2.8) of \cite{kpvcritical} with the admissible triples $(p_1,q_1,\alpha_1)=(5,10,0)$ and $(p_2,q_2,\alpha_2)=(\infty,2,1)$. In fact, if we combine (2.6) cut in time with $[0,+\infty)$ and (2.8), we get \[ {\left\| \drondx \int_t^{+\infty} W(t-t') g(x,t')\,dt' \right\|}_{L_x^5L_t^{10}} \leq C {\|g\|}_{L_x^1L_t^2}. \] If we apply it to $\tilde{g}(x,t')=\mathbbm{1}_{[T,+\infty)}(t')g(x,t')$, we reach the desired inequality since \[ \Nlcld{\drondx \int_t^{+\infty} W(t-t')g(x,t')\,dt'} \leq {\left\| \drondx \int_t^{+\infty} W(t-t') \tilde{g}(x,t')\,dt' \right\|}_{L_x^5L_t^{10}}. \qedhere \]
\end{enumerate}
\end{proof}

\subsection{Preliminary results on the linearized equation}

\subsubsection{Linearized equation}

The linearized equation appears if one considers a solution of \eqref{eq:gKdV} close to the soliton $Q(x-t)$. More precisely, if $u(t,x)=Q(x-t)+h(t,x-t)$ verifies \eqref{eq:gKdV}, then $h$ verifies \begin{equation} \label{eq:linear} \partial_t h+\L h=R(h) \end{equation} where $\L a=-(La)_x$, $La=-\partial_x^2 a+a-pQ^{p-1}a$ is defined in section \ref{subsec:linear}, and \[ R(h)=-\partial_x\left( \sum_{k=2}^p \binom{p}{k} Q^{p-k}h^k \right).\] The spectrum of $\L$ has been calculated by Pego and Weinstein in \cite{pegoweinstein}; their results are summed up here for reader's convenience.

\begin{prop}[\cite{pegoweinstein}] \label{th:spectrum}
Let $\sigma(\L)$ be the spectrum of the operator $\L$ defined on $L^2(\R)$ and let $\sigma_{\mathrm{ess}}(\L)$ be its essential spectrum. Then \[\sigma_{\mathrm{ess}}(\L)=i\R \quad \m{and}\quad \sigma(\L)\cap\R = \{-e_0,0,e_0\} \m{ with } e_0>0. \] Furthermore, $e_0$ and $-e_0$ are simple eigenvalues of $\L$ with eigenfunctions $\Y_+$ and $\Y_- = \check{\Y}_+$ which have an exponential decay at infinity, and the null space of $\L$ is spanned by $Q'$.
\end{prop}

\subsubsection{Exponential decay}

Exponential decay of $\Y_+$ has been proved in \cite{pegoweinstein}, but a generalization of this fact to a larger family of functions will be necessary in the proof of proposition \ref{th:approxsol}. For $\lambda>0$, consider the operator $A_{\lambda}$ defined on $L^2$ by $A_{\lambda} u = u'''-u'-\lambda u$, and the characteristic equation of $A_{\lambda}u=0$, \[ f_{\lambda}(x):=x^3-x-\lambda=0. \] Note $\sig1$, $\sig2$, $\sig3$ the roots of $f_{\lambda}$, eventually complex, and sorted by their real part. A simple study of $f_{\lambda}$ shows that $\sig3$ is always real, $\sig3>\frac{1}{\sqrt 3}$, and ${(\sig3)}_{\lambda>0}$ is increasing. Moreover, we have the three cases:
\begin{enumerate}[(a)]
\item If $\lambda>\frac{2}{3\sqrt3}$, then $\sig1$ and $\sig2$ are two conjugate roots which verify $\Re\sig1 =\Re\sig2=-\frac{\sig3}{2}$.
\item If $\lambda=\frac{2}{3\sqrt3}$, then $\sig1=\sig2=-\frac{1}{\sqrt3}$ and $\sig3=\frac{2}{\sqrt 3}$.
\item If $\lambda<\frac{2}{3\sqrt3}$, then $\sig1$, $\sig2$ are real and: $\sig1\in \left( -\sqrt3,-\frac{1}{\sqrt3}\right)$ ; $\sig2\in \left( -\frac{1}{\sqrt3},0\right)$. Moreover, ${(\sig2)}_{\lambda}$ is decreasing, and in particular $\sig2\nearrow 0$ when $\lambda\searrow 0$.
\end{enumerate}
This analysis allows us to define \[ \mu = \frac{1}{4} \min_{\lambda\geq e_0} (\sig3,-\Re\sig2,e_0,1) >0 \] and \[ \H = \{ f\in H^{\infty}(\R)\ |\ \forall j\geq0, \exists C_j\geq 0,\forall x\in\R,\ |f^{(j)}(x)|\leq C_je^{-\mu|x|} \}. \]

\begin{lem} \label{th:expdecayA}
If $u\in L^2$ and $f\in\H$ verify $u'''-u'-\lambda u=f$ with $\lambda\geq e_0$, then $u\in\H$.
\end{lem}

\begin{proof}
First notice that $u\in H^{\infty}(\R)$ by a simple bootstrap argument. Moreover, the method of variation of constants gives us \[ u(x) = Ae^{\sig3x} \int_x^{+\infty} e^{-\sig3s}f(s)\,ds +Be^{\sig2x} \int_{-\infty}^x e^{-\sig2s}f(s)\,ds +Ce^{\sig1x} \int_{-\infty}^x e^{-\sig1s}f(s)\,ds \] with $A,B,C\in\C$, if we suppose that $\lambda\neq \frac{2}{3\sqrt 3}$. We can also notice that $u'$ has the same form as $u$, except for three terms in $f(x)$ which appear, and which have the expected decay by hypothesis, and so on for $u^{(j)}$ for $j\geq 2$. Hence we only have to check exponential decay for $u$: \begin{multline*} |u(x)| \leq A'e^{\sig3x} \int_x^{+\infty} e^{-\sig3s}|f(s)|\,ds +B'e^{\Re\sig2x} \int_{-\infty}^x e^{-\Re\sig2s}|f(s)|\,ds\\ +C'e^{\Re\sig1x} \int_{-\infty}^x e^{-\Re\sig1s}|f(s)|\,ds. \end{multline*} By changing $x$ in $-x$ and by the definition of $\mu$, it is enough to show that if \[ v(x)=e^{-ax}\int_{-\infty}^x e^{as}e^{-\mu|s|}\,ds\] with $a\geq 2\mu$, then $v(x)\leq e^{-\mu|x|}$. Notice that one half could also have replaced one quarter in the definition of $\mu$, but this gain of $2$ allows us to treat the case $\lambda=\frac{2}{3\sqrt 3}$ (not written here for brevity), which makes appear a polynomial in front of the exponential in the last two terms of the expression of $u$. Finally, we conclude in both cases, since $a-\mu\geq \mu>0$:
\begin{itemize}
\item If $x<0$, then $v(x)\leq e^{-ax}\int_{-\infty}^x e^{as}e^{\mu s}\,ds = Ce^{-ax}\cdot e^{(a+\mu)x} = Ce^{\mu x} = Ce^{-\mu|x|}$.
\item If $x\geq 0$, then $v(x)\leq e^{-ax}\int_{-\infty}^x e^{as}e^{-\mu s}\,ds = Ce^{-ax}\cdot e^{(a-\mu)x} = Ce^{-\mu x} = Ce^{-\mu|x|}$.
\end{itemize}
The case $\lambda=\frac{2}{3\sqrt 3}$ is treated similarly.
\end{proof}

\begin{cor} \label{th:Ydecay}
$\Y_+,\Y_- \in\H$.
\end{cor}

\begin{proof}
Since $\Y_-= \check{\Y}_+$, it is enough to show that $\Y_+\in\H$. But by definition of $\Y_+$ in \cite{pegoweinstein}, we have $\L\Y_+=e_0\Y_+$ with $\Y_+\in L^2$, \emph{i.e.} \[ \Y'''_+-\Y'_+-e_0\Y_+=-p\partial_x(Q^{p-1}\Y_+) = -p(p-1)Q'Q^{p-2}\Y_+ -pQ^{p-1}\Y'_+. \] By a bootstrap argument, we have $\Y_+\in H^{\infty}(\R)$, and in particular $\Y_+^{(j)}\in L^{\infty}(\R)$ for all $j\geq 0$. If we denote $f(x)=-p(p-1)Q'Q^{p-2}\Y_+ -pQ^{p-1}\Y'_+$, then by exponential decay of $Q^{(j)}$ for all $j\geq 0$ and by definition of $\mu$, we have $|f^{(j)}(x)|\leq Ce^{-(p-1)|x|} \leq Ce^{-\mu|x|}$ and so $f\in\H$. It is enough to apply lemma \ref{th:expdecayA} with $\lambda=e_0$ to conclude.
\end{proof}

\subsection{Existence of special solutions}

We now prove the following result, which is the first part of theorem \ref{th:main}.

\begin{prop} \label{th:ua}
Let $A\in\R$. If $t_0=t_0(A)$ is large enough, then there exists a solution $U^A\in C^{\infty}\left([t_0,+\infty),H^{\infty}\right)$ of \eqref{eq:gKdV} such that \begin{equation} \label{eq:ua} \forall s\in\R, \exists C>0, \forall t\geq t_0,\quad \nhs{U^A(t,\cdot+t)-Q-Ae^{-e_0t}\Y_+}\leq Ce^{-2e_0 t}. \end{equation}
\end{prop}

\subsubsection{A family of approximate solutions}

The following proposition is similar to \cite[proposition 3.4]{duy}, except for the functional space, which is not the Schwartz space but the space $\H$ described above.

\begin{prop} \label{th:approxsol}
Let $A\in\R$. There exists a sequence ${(\mathcal{Z}_j^A)}_{j\geq 1}$ of functions of $\H$ such that $\mathcal{Z}_1^A=A\Y_+$, and if $k\geq 1$ and $\V_k^A = \sum_{j=1}^k e^{-je_0t} \mathcal{Z}_j^A$, then \begin{equation}\label{eq:approxsol} \partial_t \V_k^A +\L\V_k^A = R(\V_k^A) + \eps_k^A(t), \quad \m{where}\quad  \eps_k^A(t)=\sum_{j=k+1}^{pk} e^{-je_0t} g_{j,k}^A,\quad g_{j,k}^A\in\H, \end{equation} and $R$ is defined in \eqref{eq:linear}.
\end{prop}

\begin{proof}
The proof is very similar to the one in \cite{duy}, and we write it there for reader's convenience. We prove this proposition by induction, and for brevity, we omit the superscript $A$.

Define $\Z_1 := A\Y_+$ and $\V_1 := e^{-e_0t}\Z_1$. Then by the explicit definition of $R$ in \eqref{eq:linear}, \[ \partial_t\V_1 +\L\V_1 -R(\V_1) = -R(\V_1) = -R(Ae^{-e_0t}\Y_+) = \sum_{j=2}^p e^{-je_0t} A^j \binom{p}{j} \partial_x [Q^{p-j}\Y_+^j]  \] which yields \eqref{eq:approxsol} for $k=1$, since $\Y_+,Q\in\H$ by corollary \ref{th:Ydecay} and claim \ref{th:solitons}.

Let $k\geq 1$ and assume that $\Z_1,\ldots,\Z_k$ are known with the corresponding $\V_k$ satisfying \eqref{eq:approxsol}. Now let $\U_{k+1} := g_{k+1,k}\in\H$, so that \[ \partial_t \V_k +\L\V_k = R(\V_k) +e^{-(k+1)e_0t}\U_{k+1} + \sum_{j=k+2}^{pk} e^{-je_0t} g_{j,k}, \] and define $\Z_{k+1} := -{(\L-(k+1)e_0)}^{-1}\U_{k+1}$. Note that $\Z_{k+1}$ is well defined since $(k+1)e_0$ is not in the spectrum of $\L$ by proposition \ref{th:spectrum}, and moreover $\Z_{k+1}\in\H$. Indeed, we have \[ \Z'''_{k+1} -\Z'_{k+1} -(k+1)e_0\Z_{k+1} = -\U_{k+1} -p(p-1)Q'Q^{p-2}\Z_{k+1} -pQ^{p-1}\Z'_{k+1} \in\H \] by exponential decay of $Q^{(j)}$ for all $j\geq 0$ and since $\Z_{k+1}^{(j)} \in H^{\infty}(\R)$ by a bootstrap argument. Hence $\Z_{k+1}\in\H$ by lemma \ref{th:expdecayA} applied with $\lambda=(k+1)e_0\geq e_0$.

Then we have \[ \partial_t\left( \V_k +e^{-(k+1)e_0t}\Z_{k+1} \right) +\L\left( \V_k+e^{-(k+1)e_0t}\Z_{k+1}\right) = R(\V_k) + \sum_{j=k+2}^{pk} e^{-je_0t} g_{j,k}. \] Denote $\V_{k+1} := \V_k +e^{-(k+1)e_0t}\Z_{k+1}$. Thus we have \[ \partial_t\V_{k+1} +\L\V_{k+1} -R(\V_{k+1}) = R(\V_k)-R(\V_{k+1}) + \sum_{j=k+2}^{pk} e^{-je_0t} g_{j,k}. \] We conclude the proof by evaluating \begin{align*} R(\V_k)-R(\V_{k+1}) &= R(\V_k)-R(\V_k+e^{-(k+1)e_0t}\Z_{k+1})\\ &= \partial_x \left[ \sum_{j=2}^p \binom{p}{j} Q^{p-j} \left( {(\V_k+e^{-(k+1)e_0t}\Z_{k+1})}^j -\V_k^j \right) \right] = \sum_{j=k+2}^{p(k+1)} e^{-je_0t}\wt{g}_{j,k}, \end{align*} which yields \eqref{eq:approxsol} for $k+1$, and thus completes the proof.
\end{proof}

\subsubsection{Construction of special solutions} \label{subsubsec:construction}

We now prove proposition \ref{th:ua}, following the same three steps as in \cite{duy}. The main difference comes from step 2, because of the derivative in the error term which forces us to use the sharp smoothing effect developed in \cite{kpv}. Let $A\in\R$ and $s\geq 1$ integer. Write \[U^A(t,x+t)=Q(x)+h^A(t,x).\] First, by a fixed point argument, we construct a solution $h^A\in C^0([t_k,+\infty),H^s)$ of \eqref{eq:linear} for $k$ and $t_k$ large and such that \begin{equation} \label{eq:hv} \forall T\geq t_k,\ \nhs{(h^A-\V_k)(T)}\leq e^{-(k+\frac{1}{2})e_0T}.\end{equation} Next, the same arguments like in \cite{duy} show that $h^A$ does not depend on $s$ and $k$. For brevity, we omit the superscript $A$.\bigskip

\emph{Step 1. Reduction to a fixed point problem.} If we set $\tilde{h}(t,x)=h(t,x-t)$, equation \eqref{eq:linear} can be written as \begin{equation} \label{eq:linearbis} \partial_t \tilde{h}+\partial_x^3 \tilde{h} =-S(\tilde{h}),\quad S(\tilde{h})=\drondx \left[ \sum_{k=1}^p \binom{p}{k} Q^{p-k}(x-t)\tilde{h}^k\right]. \end{equation} Moreover, we have by \eqref{eq:approxsol}, $\eps_k(t)=\partial_t\V_k+\partial_x^3\V_k-\partial_x\V_k +\partial_x \left[ \sum_{j=1}^p \binom{p}{j} Q^{p-j}\V_k^j \right]$. Now let $v(t,x)=(h-\V_k)(t,x-t)$ and subtract the previous equation from \eqref{eq:linearbis}, so that \[ \partial_t v+\partial_x^3 v = -S[v+\V_k(t,x-t)]+S[\V_k(t,x-t)]-\eps_k(t,x-t).\] For notation simplicity, we drop the space argument $(x-t)$ for the moment. The equation can then be written as \begin{equation} v(t)=\M(v)(t) := \int_t^{+\infty} W(t-t')\left[ S(\V_k(t')+v(t'))-S(\V_k(t'))+\eps_k(t')\right]\,dt'. \end{equation} Note that \eqref{eq:hv} is equivalent to $\nhs{v(T)}\leq e^{-(k+\frac{1}{2})e_0T}$ for $T\geq t_k$. In other words, defining \[ \left\{ \begin{aligned} N_1(v) &= \sup_{T\geq t_k} e^{(k+\frac{1}{2})e_0T} \nhs{v(T)},\\ N_2(v) &= \sum_{s'=0}^s \sup_{T\geq t_k} e^{(k+\frac{1}{2})e_0T} \nlcld{\partial^{s'}v},\\ \Lambda(v) &= \Lambda_{t_k,k,s}(v) = \max(N_1(v),N_2(v)), \end{aligned} \right. \] it is enough to show that $\M$ is a contraction on $B$ defined by \[ B=B(t_k,k,s)=\left\{ v\in C^0([t_k,+\infty),H^s)\ |\ \Lambda(v)\leq 1 \right\}. \]

\begin{rem} \label{th:lcld}
The choice of the two norms $N_1$ and $N_2$ is related to the fact that global well posedness of supercritical \eqref{eq:gKdV} with initial data small in $H^1$ can be proved with the two norms $\wt{N_1}(v)=\sup_{t\in\R} \nh{v(t)}$ and $\wt{N_2}(v) = {\|v\|}_{L_x^5L_t^{10}} + {\|\partial_x v\|}_{L_x^5L_t^{10}}$, following \cite{kpvcritical}. We could also have used other norms from \cite{kpv}.\smallskip
\end{rem}

\emph{Step 2. Contraction argument.} We show that $\M$ is a contraction on $B$ for $s\geq 1$ and $k,t_k$ sufficiently large. Throughout this proof, we denote by $C$ a constant depending only on $s$, and $C_k$ a constant depending on $s$ and $k$. To estimate $N_1(\M(v))$ and $N_2(\M(v))$, we have to explicit \begin{align*} S(\V_k+v)-S(\V_k) &= \drondx \left[ \sum_{i=1}^p \binom{p}{i} Q^{p-i}\left( \Puiss{\V_k+v}{i}-\V_k^i\right) \right]\\ &= \drondx\left( pQ^{p-1}v\right) +\drondx \left[ \sum_{i=2}^p \binom{p}{i} Q^{p-i}v\cdot\sum_{l=1}^i \binom{i}{l} \V_k^{i-l}v^{l-1} \right]\\ &= p\frac{\partial \mathbf{I}}{\partial x}+ \sum_{\abc} C_{\abc} \frac{\partial \mathbf{II}_{\abc}}{\partial x} \end{align*} where $\mathbf{I}=Q^{p-1}v$ and $\mathbf{II}_{\abc}=Q^{\alpha}\V_k^{\beta}v^{\gamma}$, with: $\gamma\geq 1$, $\beta+\gamma\geq 2$, $\alpha+\beta+\gamma=p\geq 6$. We can now write \begin{multline*} \partial^s \M(v) = p\int_t^{+\infty} W(t-t')\drondx\left[\partial^s(\mathbf{I})\right]\,dt' +\sum_{\abc} C_{\abc} \int_t^{+\infty} W(t-t')\drondx\left[ \partial^s (\mathbf{II}_{\abc}) \right]\,dt'\\ +\int_t^{+\infty} W(t-t')\partial^s \eps_k(t')\,dt'. \end{multline*} By \eqref{eq:kpv1} and \eqref{eq:kpv2}, we obtain \begin{multline} \label{eq:ptfixe} \max\left({\|\partial^s \M(v)(T)\|}_{L_x^2},\nlcld{\partial^s \M(v)}\right) \leq C\nluld{\partial^{s-1}\eps_k} + C\nluld{\partial^s(\mathbf{I})}\\ +\sum_{\abc} C_{\abc} \nluld{\partial^s(\mathbf{II}_{\abc})}. \end{multline} We treat the terms $\eps_k$, $\mathbf{I}$, $\mathbf{II}_{\abc}$ for $\alpha=p-2,\beta=\gamma=1$, and for $\alpha=\beta=0,\gamma=p$. All other terms can be treated similarly: for example, $\mathbf{II}_{0,p-1,1}$ can be treated like $\mathbf{II}_{p-2,1,1}$, etc.\bigskip

For $\mathbf{I}$, since $Q$ and his derivatives have the same decay, it is enough to estimate the term $\tilde{\mathbf{I}} = \nluld{Q^{p-1}\partial^s v} \leq C\nluld{e^{-|x-t|}\partial^s v}$: \begin{align*} \tilde{\mathbf{I}} &\leq C{\|e^{x-t}\partial^s v\|}_{L_{(-\infty,T]}^1L_{[T,+\infty)}^2} +C{\|e^{t-x}\partial^s v\|}_{L_{[T,+\infty)}^1L_{[T,x]}^2} +C{\|e^{x-t}\partial^s v\|}_{L_{[T,+\infty)}^1L_{[x,+\infty)}^2}\\ &\leq C\sqrt{\int_{-\infty}^T e^{2x}\,dx}\sqrt{\int_x\int_T^{+\infty} e^{-2t}\Carre{\partial^s v}\,dt\,dx} +C\sqrt{\int_T^{+\infty} e^{-2x}\,dx} \sqrt{\int_x\int_T^{+\infty} e^{2t}\Carre{\partial^s v}\,dt\,dx}\\ &\qquad +C\sqrt{\int_T^{+\infty} e^{-2x}\,dx} \sqrt{\int_T^{+\infty} \int_x^{+\infty} e^{4x-2t}\Carre{\partial^s v}\,dt\,dx} \end{align*} by the Cauchy-Schwarz inequality. Now, by Fubini's theorem, and since $4x-2t\leq 2t$ in the last integral, we get \begin{align*} \tilde{\mathbf{I}} &\leq Ce^T N_1(v)\sqrt{\int_T^{+\infty} e^{-(2k+1)e_0t-2t}\,dt} + 2Ce^{-T}N_1(v)\sqrt{\int_T^{+\infty} e^{-(2k+1)e_0t+2t}\,dt}\\ &\leq Ce^T N_1(v)\frac{e^{-(k+\frac{1}{2})e_0T-T}}{\sqrt{(2k+1)e_0+2}} +2Ce^{-T}N_1(v)\frac{e^{-(k+\frac{1}{2})e_0T +T}}{\sqrt{(2k+1)e_0-2}} \leq CN_1(v)\frac{1}{\sqrt k}e^{-(k+\frac{1}{2})e_0T}. \end{align*} Note that since $k$ will be chosen large at the end of the argument, we can suppose $(2k+1)e_0>2$.

For $\mathbf{II}_{p-2,1,1}$, we treat similarly the term $\wt{\mathbf{II}}_{p-2,1,1} = \nluld{Q^{p-2}\V_k\partial^s v}$ since $\V_k$ and his derivatives have the same decay. In fact, we have by Hölder inequality \[ \wt{\mathbf{II}}_{p-2,1,1} \leq C\nlcld{\partial^sv} {\|\V_k\|}_{L_x^{5/4}L_{[T,+\infty)}^{5/2}} \leq CN_2(v)e^{-(k+\frac{1}{2})e_0T} {\|\V_k\|}_{L_x^{5/4}L_{[T,+\infty)}^{5/2}}. \] By the definition of $\V_k$ in proposition \ref{th:approxsol}, we have by noting $e'_0 = \frac{5}{2}e_0$ and $\mu' = \frac{5}{2}\mu$, \begin{align*} {\|\V_k\|}_{L_x^{5/4}L_{[T,+\infty)}^{5/2}}^{5/4} &\leq C_k {\| e^{-e_0t} e^{-\mu|x-t|} \|}_{L_x^{5/4} L_{[T,+\infty)}^{5/2}}^{5/4}\\ &\leq C_k\int_{-\infty}^T \sqrt{ \int_T^{+\infty} e^{-e'_0t}e^{-\mu't}e^{\mu' x}\,dt}\,dx + C_k\int_T^{+\infty} \sqrt{\int_T^x e^{-e'_0t}e^{\mu't}e^{-\mu'x}\,dt}\,dx\\ &\qquad + C_k\int_T^{+\infty} \sqrt{\int_x^{+\infty} e^{-e'_0t}e^{-\mu't}e^{\mu'x}\,dt}\,dx\\ &\leq C_k e^{\frac{\mu'}{2}T} \sqrt{\int_T^{+\infty} e^{-(e'_0+\mu')t}\,dt} + C_k e^{-\frac{\mu'}{2}T} \sqrt{\int_T^{+\infty} e^{(\mu'-e'_0)t}\,dt}\\ &\qquad +C_k \int_T^{+\infty} e^{\frac{\mu'}{2}x} \sqrt{\int_x^{+\infty} e^{-(e'_0+\mu')t}\,dt}\,dx\\ &\leq 3C_k e^{-\frac{e'_0}{2}T}\quad \m{ since $\mu<e_0$ by definition of $\mu$.} \end{align*} We finally deduce that ${\|\V_k\|}_{L_x^{5/4}L_{[T,+\infty)}^{5/2}} \leq C_k e^{-e_0T}$ and so $\wt{\mathbf{II}}_{p-2,1,1} \leq C_k N_2(v) e^{-(k+\frac{3}{2})e_0T}$.

For $\mathbf{II}_{0,0,p} = v^p$, first remark that \[ \partial^s(v^p) = p\partial^{s-1}(\partial v\cdot v^{p-1}) = p\partial^sv\cdot v^{p-1} + p\sum_{k=0}^{s-2} \binom{s-1}{k} \partial^{k+1} v\cdot \partial^{s-1-k}(v^{p-1}) \] where each term of the sum is a product of $p$ terms like $\partial^{s_j}v$ with $s_j\leq s-1$. Since $H^1(\R) \hookrightarrow L^{\infty}(\R)$, we can estimate the first term thanks to Hölder's inequality: \begin{align*} \nluld{\partial^sv\cdot v^{p-1}} &\leq {\|v\|}_{L_{[T,+\infty)}^{\infty}L_x^{\infty}}^{p-5}\cdot \nlcld{\partial^s v} \cdot \nlcld{v}^4\\ &\leq Ce^{-p(k+\frac{1}{2})e_0T} {N_1(v)}^{p-5} {N_2(v)}^5. \end{align*} The other terms in the sum can be treated in the same way, and more simply since we can choose any $(p-5)$ terms to take out in $L_{[T,+\infty)}^{\infty} L_x^{\infty}$ norm, and any $5$ others left in $L_x^5L_{[T,+\infty)}^{10}$ norm.

For $\eps_k$, we deduce by a similar calculation like above and by the expression of $\eps_k$ in \eqref{eq:approxsol} that \[ \nluld{\partial^{s-1} \eps_k} \leq C_k \int_{\R} \sqrt{ \int_T^{+\infty} e^{-2(k+1)e_0t}e^{-2\mu|x-t|}\,dt}\,dx \leq C'_k e^{-(k+1)e_0T}. \]

Summarizing from \eqref{eq:ptfixe}, we have shown \begin{multline*} \max\left( e^{(k+\frac{1}{2})e_0T}\nhs{\M(v)(T)}, \sum_{s'=0}^s e^{(k+\frac{1}{2})e_0T} \nlcld{\partial^{s'} v} \right)\\ \leq C_ke^{-\frac{e_0}{2}T} + \frac{CN_1(v)}{\sqrt k} + C_kN_2(v)e^{-e_0T}+ Ce^{-(p-1)(k+\frac{1}{2})e_0T} \puiss{N_1(v)}{p-5} \puiss{N_2(v)}{5}. \end{multline*} Since $v\in B(t_k,k,s)$, \emph{i.e.} $\Lambda(v)\leq 1$, then we have \[ \Lambda(\M(v)) \leq C_ke^{-\frac{e_0}{2}t_k}+ \left( \frac{C}{\sqrt k}+C_ke^{-e_0t_k} \right) \Lambda(v) \leq \left( \frac{C}{\sqrt k}+C_ke^{-\frac{e_0}{2}t_k} \right). \] First, choose $k$ so that $\frac{C}{\sqrt k}\leq \frac{1}{2}$, then take $t_k$ such that $C_ke^{-\frac{e_0}{2}t_k}\leq \frac{1}{2}$. Then $\M$ maps $B=B(t_k,k,s)$ into itself.

It remains to show that $\M$ is a contraction on $B$. If we $v,w\in B$, we have \[ \M(v)-\M(w) = \int_t^{+\infty} W(t-t') \Big[ S(\V_k(t')+v(t'))-S(\V_k(t')+w(t')) \Big] dt' \] and \begin{align*} S(\V_k+v)-S(\V_k+w) &= \drondx \left[ \sum_{j=1}^p \binom{p}{j} Q^{p-j} \left[ \Puiss{\V_k+v}{j}-\Puiss{\V_k+w}{j} \right] \right]\\ &= \drondx \sum_{j=1}^p \binom{p}{j} Q^{p-j}(v-w)\sum_{i=1}^{j-1} \Puiss{\V_k+v}{i}\Puiss{\V_k+w}{j-i}\\ &= p\drondx \Big[ Q^{p-1}(v-w)\Big] + \drondx \Big[ (v-w)\cdot\!\! \sum_{\abcd} C_{\abcd} Q^{\alpha}\V_k^{\beta}v^{\gamma}w^{\delta} \Big]. \end{align*} Under this form, a similar calculation like above allows us to conclude: the first term is treated like $\mathbf{I}$, and each $Q^{\alpha}\V_k^{\beta}v^{\gamma}w^{\delta}$ can be treated like $\mathbf{II}_{\abc}$ if we systematically take out the term $\Lambda(v-w)$ by Hölder's inequality. Hence we get, since there is no term in $\eps_k$, \[ \Lambda(\M(v)-\M(w)) \leq \left( \frac{C}{\sqrt k}+C_ke^{-e_0t_k} \right) \Lambda(v-w). \] Choosing if necessary a larger $k$, then a larger $t_k$, we may assume that $\frac{C}{\sqrt k}<\frac{1}{2}$ and $C_ke^{-e_0t_k}\leq \frac{1}{2}$, showing that $\M$ is a contraction on $B$. Hence, step 2 is complete.\bigskip

\emph{Step 3. End of the proof.} By the preceding step with $s=1$, there exist $k_0$ and $t_0$ such that there exists a unique solution $U^A$ of \eqref{eq:gKdV} satisfying $U^A \in C^0([t_0,+\infty),H^1)$ and \begin{equation} \label{eq:uas2} \Lambda_{t_0,k_0,1}\Big(U^A(t,x)-Q(x-t)-\V_{k_0}^A(t,x-t)\Big) \leq 1. \end{equation} Note that the fixed point argument still holds taking a larger $t_0$, and so the uniqueness remains valid, for any $t'_0\geq t_0$, in the class of solutions of \eqref{eq:gKdV} in $C^0([t'_0,+\infty),H^1)$ satisfying \eqref{eq:uas2}.

Finally, we can show proposition \ref{th:ua}. Since $U^A$ is a solution of \eqref{eq:gKdV}, it is sufficient to show that $U^A\in C^0([t_0,+\infty),H^s)$ for any $s$; the smoothness in time will follow from the equation. Let $s\geq 1$: by step 2, if $k_s$ is large enough, there exist $t_s$ and $\wt{U}^A \in C^0([t_s,+\infty),H^s)$ such that \[ \Lambda_{t_s,k_s,s}\Big(\wt{U}^A(t,x)-Q(x-t)-\V_{k_s}^A(t,x-t)\Big) \leq 1. \] Of course, we may choose $k_s\geq k_0+1$. But by construction of $\V_k^A$ in proposition \ref{th:approxsol}, we have $\V_{k_s}^A(t,x-t) -\V_{k_0}^A(t,x-t) = \sum_{j=k_0+1}^{k_s} e^{-je_0t} \Z_j^A(x-t)$ where $\Z_j^A\in\H$, and so by similar calculation like in step 2, \[ \Lambda_{t_s,k_0,s}\Big(\V_{k_s}^A(t,x-t)-\V_{k_0}^A(t,x-t)\Big) \leq Ce^{-\frac{e_0}{2}t_s} \leq \frac{1}{2} \] for $t_s$ large enough. Moreover, we have by definition of $\Lambda$ (and since $k_0\leq k_s-1$) \[ \Lambda_{t_s,k_0,s}(u) \leq e^{-e_0t_s}\Lambda_{t_s,k_s,s}(u). \] Thus, if we choose $t_s$ large enough such that $e^{-e_0t_s}\leq\frac{1}{2}$, we get by triangle inequality \begin{multline*} \Lambda_{t_s,k_0,1} \Big( \wt{U}^A(t,x)-Q(x-t)-\V_{k_0}^A(t,x-t) \Big) \leq \Lambda_{t_s,k_0,s} \Big( \wt{U}^A(t,x)-Q(x-t)-\V_{k_0}^A(t,x-t) \Big)\\ \leq \Lambda_{t_s,k_0,s} \Big( \wt{U}^A(t,x)-Q(x-t)-\V_{k_s}^A(t,x-t) \Big) + \Lambda_{t_s,k_0,s}\Big(\V_{k_s}^A(t,x-t)-\V_{k_0}^A(t,x-t)\Big) \leq 1. \end{multline*} In particular, $\wt{U}^A$ satisfies \eqref{eq:uas2} for large $t_s$. By the uniqueness in the fixed point argument, we have $U^A=\wt{U}^A$, which shows that $U^A \in C^0([t_s,+\infty),H^s)$. By the persistence of regularity of \eqref{eq:gKdV} equation, $U^A\in C^0([t_0,+\infty),H^s)$, where $s\geq 1$. In particular, by compactness on $[t_0,t_s]$, there exists $C=C(s)$ such that \[ \forall t\geq t_0,\quad \nhs{U^A(t,x)-Q(x-t) -\V_{k_0}^A(t,x-t)} \leq Ce^{-(k_0+\frac{1}{2})e_0t} \] and so \eqref{eq:ua} follows, which achieves the proof of proposition \ref{th:ua}.

\subsection{Uniqueness}

Now, the special solution $U^A$ being constructed, we prove its uniqueness, in the sense of the following proposition, which implies the second part of theorem \ref{th:main}.

\begin{prop} \label{th:uniqueness}
Let $u$ be a solution of \eqref{eq:gKdV} such that \begin{equation} \label{eq:condition} \inf_{y\in\R} \nh{u(t)-Q(\cdot-y)}\xrightarrow[t\to+\infty]{}0. \end{equation} Then there exist $A\in\R$, $t_0\in\R$ and $x_0\in\R$ such that $u(t)=U^A(t,\cdot-x_0)$ for all $t\geq t_0$, where $U^A$ is the solution of \eqref{eq:gKdV} defined in proposition \ref{th:ua}.
\end{prop}

The proof of proposition \ref{th:uniqueness} proceeds in four steps: first we improve condition \eqref{eq:condition} into an exponential convergence and we control the translation parameter, then we improve the exponential convergence up to any order, and finally we adapt step 3 of \cite{duy} to \eqref{eq:gKdV} to conclude the proof. A crucial argument for the first and third steps is the coercivity of $(L\cdot,\cdot)$ under orthogonality to eigenfunctions of the adjoint of $\L$, proved in \cite{martel:Nsolitons}.

\subsubsection{Adjoint of $\L$}

We recall that $L$ is defined by $La = -\partial_x^2 a +a-pQ^{p-1}a$ and $\L$ by $\L = -\partial_x L$. In particular, the adjoint of $\L$ is $L\partial_x$. Moreover $\L$ has two eigenfunctions $\Y_{\pm}$, with $\L \Y_{\pm} = \pm e_0 \Y_{\pm}$ where $e_0>0$.

\begin{lem} \label{th:Z}
Let $Z_{\pm} = L\Y_{\pm}$. Then the following properties hold:
\begin{enumerate}[(i)]
\item $Z_{\pm}$ are two eigenfunctions of $L\partial_x$: $L(\partial_x Z_{\pm}) = \mp e_0 Z_{\pm}$.
\item $(\Y_+,Z_+)=(\Y_-,Z_-)=0$ and $(Z_+,Q')=(Z_-,Q')=0$.
\item There exists $\sigma_1>0$ such that, for all $v\in H^1$ such that $(v,Z_+)=(v,Z_-)=(v,Q')=0$, $(Lv,v)\geq \sigma_1\nh{v}^2$.
\item One has $(\Y_+,Z_-)\neq 0$ and $(Q',\Y'_+)\neq 0$. Hence one can normalize $\Y_{\pm}$ and $Z_{\pm}$ to have \[(\Y_+,Z_-)=(\Y_-,Z_+)=1,\quad (Q',\Y'_+)>0\quad \m{and still}\quad L\Y_{\pm}=Z_{\pm}. \]
\item There exist $\sigma_2>0$ and $C>0$ such that for all $v\in H^1$, \begin{equation} \label{eq:coerbis} (Lv,v)\geq \sigma_2 \nh{v}^2 -C\carre{(v,Z_+)} -C\carre{(v,Z_-)} -C\carre{(v,Q')}. \end{equation}
\end{enumerate}
\end{lem}

\begin{proof}
\begin{enumerate}[(i)]
\item It suffices to apply $L$ to the equality $-\partial_x(L\Y_{\pm})=\pm e_0\Y_{\pm}$.
\item We have $(\Y_{\pm},Z_{\pm}) = \mp \frac{1}{e_0}(\partial_x(L\Y_{\pm}),L\Y_{\pm})=0$ and $(Z_{\pm},Q') = (L\Y_{\pm},Q') = (\Y_{\pm},LQ') = 0$ since $LQ'=0$ and $L$ is self-adjoint.
\item This fact is assertion (7) proved in \cite{martel:Nsolitons}.
\item If we had $(\Y_+,Z_-)=(Z_+,\Y_-)=0$, then by (ii) we would have in fact $(\Y_++\Y_-) \bot Z_+,Z_-,Q'$ since $Q'$ is odd and $\Y_++\Y_-$ is even, and so by (iii): $(L(\Y_++\Y_-),\Y_++\Y_-)\geq \sigma_1 \nh{\Y_++\Y_-}^2$.  But $(L(\Y_++\Y_-),\Y_++\Y_-) = (L\Y_+,\Y_+)+(L\Y_-,\Y_-)+2(L\Y_+,\Y_-) = (Z_+,\Y_+)+(Z_-,\Y_-)+2(Z_+,\Y_-)=0$, and so we would get $\nh{\Y_++\Y_-}=0$, \emph{i.e.} $\Y_+ = -\Y_-$, which is a contradiction with the independence of the family $(\Y_+,\Y_-)$.

    Similarly, if we had $(Q',\Y'_+)= 0$, we would have $(Q'',\Y_+)=0$. Moreover we have $(Q,\Y_+)=-\frac{1}{e_0}(Q,(L\Y_+)') = \frac{1}{e_0}(LQ',\Y_+)=0$, and so we would have \[ (Q,Z_+) = (Q,L\Y_+) = (LQ,\Y_+) = (-Q''+Q-pQ^p,\Y_+) = -p(Q-Q'',\Y_+)=0. \] But we would also have $(Q,Z_-)=0$ since $Q$ is even and $Z_- = \check{Z}_+$. Since $(Q,Q')=0$, we would finally have $(LQ,Q)\geq \sigma_1 \nh{Q}^2$ by (iii). But a straightforward calculation gives $(LQ,Q)=(1-p)\int Q^{p+1} <0$, and so a contradiction.

    Finally, if we note $\eta=(\Y_+,Z_-)\neq 0$, then the normalization $\wt{\Y_-} = \frac{1}{\eta}\Y_-$, $\wt{Z_-} = \frac{1}{\eta}Z_- = L\wt{\Y_-}$ satisfies the required properties if $(Q',\Y'_+)>0$. Otherwise, it suffices to change $\Y_{\pm}$ and $Z_{\pm}$ in $-\Y_{\pm}$ and $-Z_{\pm}$ respectively.
\item Let $v\in H^1$, and decompose it as \[ v = \alpha\Y_+ +\beta \Y_- +\gamma Q' +v_{\bot} \] with $\alpha=(v,Z_-)$, $\beta=(v,Z_+)$, $\gamma = \nld{Q'}^{-2} [(v,Q')-\alpha(\Y_+,Q')-\beta(\Y_-,Q')]$ and $v_{\bot}$ orthogonal to $Z_+,Z_-,Q'$ by the previous normalization. We have by straightforward calculation $(Lv,v) = (Lv_{\bot},v_{\bot})+2\alpha\beta$, and $(Lv_{\bot},v_{\bot}) \geq \sigma_1\nh{v_{\bot}}^2$ by (iii), so we have $(Lv,v)\geq \sigma_1\nh{v_{\bot}}^2 -\alpha^2-\beta^2$. Finally, we have by the previous decomposition of $v$ that \[ \nh{v}^2 \leq C(\alpha^2+\beta^2 +\gamma^2 +\nh{v_{\bot}}^2) \leq C'(\alpha^2+\beta^2+ \carre{(v,Q')}+\nh{v_{\bot}}^2) \] and so $(Lv,v) \geq \sigma_1\left[ \frac{\nh{v}^2}{C'}-\alpha^2-\beta^2-\carre{(v,Q')} \right] -\alpha^2-\beta^2$, as desired. \qedhere
\end{enumerate}
\end{proof}

\subsubsection{Step 1: Improvement of the decay at infinity}

We begin the proof of proposition \ref{th:uniqueness} here: let $u$ be a solution of \eqref{eq:gKdV} verifying \eqref{eq:condition}.

\begin{itemize}
\item By lemma \ref{th:decomp}, we can write $\eps(t,x)=u(t,x+x(t))-Q(x)$ for $t\geq t_0$ with $t_0$ large enough, where $\eps$ verifies $\nh{\eps(t)}\vers 0$ and $\eps(t)\bot Q'$ for all $t\geq t_0$. We recall that we have by proposition \ref{th:eqeps}: \begin{equation} \label{eq:eqeps} \eps_t -{(L\eps)}_x = (x'-1){(Q+\eps)}_x+R(\eps) \end{equation} where $\nlu{R(\eps)}\leq C\nh{\eps}^2$ and $|x'-1|\leq C\nh{\eps}$.
\item Now consider \[ \alpha_+(t) = \int Z_+ \eps(t),\qquad \alpha_-(t) = \int Z_-\eps(t) \] where $Z_{\pm}$ are defined in lemma \ref{th:Z}. Since $\nh{\eps(t)}\vers 0$, we have of course $\alpha_{\pm}(t)\vers 0$. The two remaining points will be to show that $\alpha_{\pm}(t)$ control $\nh{\eps(t)}$, and have exponential decay at infinity.
\item First, we recall the linearization of Weinstein's functional (lemma \ref{th:Flin}): \[ F(Q+\eps) = F(Q) +\frac{1}{2}(L\eps,\eps)+K(\eps) \] where $|K(\eps)|\leq C\nh{\eps}^3$. But $F(Q+\eps)-F(Q)$ is a constant which tends to $0$ at infinity in time, and so is null, hence we get $|(L\eps,\eps)|\leq C\nh{\eps}^3$. We now use \eqref{eq:coerbis}, which gives since $(\eps,Q')=0$: \[ (L\eps,\eps) \geq \sigma_2 \nh{\eps(t)}^2 -C\alpha_+^2(t) -C\alpha_-^2(t) \] and so $\sigma_2\nh{\eps(t)}^2 -C\alpha_+^2(t)-C\alpha_-^2(t)-C'\nh{\eps(t)}^3\leq 0$. For $t_0$ chosen possibly larger, we conclude that \[ \nh{\eps(t)}^2 \leq C(\alpha_+^2(t)+\alpha_-^2(t)). \]
\item We have now to obtain exponential decay of $\alpha_{\pm}$ to conclude the first step. If we multiply \eqref{eq:eqeps} by $Z_+$ and integrate, we obtain \[ \alpha'_+(t)-e_0\alpha_+(t) = (x'-1)\int {(Q+\eps)}_x Z_+ +\int R(\eps)Z_+ = (x'-1)\int \eps_x Z_+ +\int R(\eps)Z_+ \] by integrating by parts and using (i) and (ii) of lemma \ref{th:Z}. By the controls of $|x'-1|$ and $R(\eps)$, we get $|\alpha'_+-e_0\alpha_+|\leq C\nh{\eps}^2 \leq C(\alpha_+^2+\alpha_-^2)$. Doing similarly with $Z_-$, we have finally the differential system \begin{empheq}[left=\empheqlbrace\,]{align} |\alpha'_+-e_0\alpha_+| &\leq C(\alpha_+^2+\alpha_-^2),\\ |\alpha'_-+e_0\alpha_-| &\leq C(\alpha_+^2+\alpha_-^2). \label{eq:alpham} \end{empheq}
\item Now define $h(t)=\alpha_+(t)-M\alpha_-^2(t)$ where $M$ is a large constant to define later. Multiplying \eqref{eq:alpham} by $|\alpha_-|$ (which can of course be taken less than $1$), we get \begin{align*} h'(t) &= \alpha'_+(t)-2M\alpha_-(t)\alpha'_-(t) \geq e_0\alpha_+ -C(\alpha_+^2+\alpha_-^2)+2Me_0\alpha_-^2 -2CM|\alpha_-|(\alpha_+^2+\alpha_-^2)\\ &\geq e_0h+3Me_0\alpha_-^2 -2Ch^2-2CM^2\alpha_-^4 -C^*\alpha_-^2 -4CMh^2-4CM^3{|\alpha_-|}^5-2CM{|\alpha_-|}^3 \end{align*} since $\alpha_+^2 = \Carre{h+M\alpha_-^2} \leq 2(h^2+M^2\alpha_-^4)$. We now fix $M= \frac{C^*}{e_0}$, so that \[ h' \geq e_0h -2Ch^2-4CMh^2 +\alpha_-^2\left( 2Me_0-2CM^2\alpha_-^2-4CM^3{|\alpha_-|}^3-2CM|\alpha_-|\right). \] Then for $t$ large enough, the expression in parenthesis is positive, and so \[ h'\geq e_0h -c_Mh^2. \] Now take $t_0$ large enough such that for $t\geq t_0$, we have $c_Mh^2\leq \frac{e_0}{2}|h|$, and suppose for the sake of contradiction that there exists $t_1\geq t_0$ such that $h(t_1)>0$. Define $T=\sup\{ t\geq t_1\ |\ h(t)>0\}$ and suppose that $T<+\infty$: since $h'(t)\geq e_0\left( h(t)-\frac{|h(t)|}{2} \right)$ for all $t\geq t_0$ and of course $h(T)=0$, we would have in particular $h'(T)\geq 0$, so $h$ increasing near $T$, and so $h(t)\leq 0$ for $t\in [T-\eps,T]$, which would be in contradiction with the definition of $T$. Hence we have $T=+\infty$, and so $h(t)>0$ for all $t\geq t_1$. Consequently, we would have $h'(t)\geq \frac{e_0}{2}h(t)$ for all $t\geq t_1$, and so $h(t)\geq Ce^{\frac{e_0}{2}t}$, which would be a contradiction with $\lim_{t\to +\infty} h(t) = 0$. Therefore we have $h(t)\leq 0$ for all $t\geq t_0$. Since $-\alpha_+$ satisfies the same differential system, we obtain by the same technique: $\forall t\geq t_0, |\alpha_+(t)|\leq M\alpha_-^2(t)$.
\item Reporting this estimate in \eqref{eq:alpham}, we obtain \[ |\alpha'_-(t)+e_0\alpha_-(t)|\leq C\alpha_-^2(t) \leq \frac{e_0}{10}|\alpha_-(t)| \] for $t$ large enough. In other words, we have $|(e^{e_0t}\alpha_-(t))'|\leq \frac{e_0}{10} |e^{e_0t}\alpha_-(t)|$, and so by integration: $|\alpha_-(t)|\leq Ce^{-\frac{9}{10}e_0t}$. By a bootstrap argument we get $|\alpha'_-(t)+e_0\alpha_-(t)|\leq Ce^{-\frac{9}{10}e_0t}|\alpha_-(t)|$, and so still by integration, we get $|e^{e_0t}\alpha_-(t)|\leq C$ for all $t\geq t_0$, \emph{i.e.} $|\alpha_-(t)|\leq Ce^{-e_0t}$. By the previous point, we also obtain \begin{equation} \label{eq:alphap} |\alpha_+(t)|\leq Ce^{-2e_0t} \end{equation} and finally $\nh{\eps(t)}^2 \leq C(\alpha_+^2(t)+\alpha_-^2(t)) \leq Ce^{-2e_0t}$.
\end{itemize}

For clarity, we summarize the results obtained so far.

\begin{lem} \label{th:improve}
If $u$ is a solution of \eqref{eq:gKdV} which verifies $\inf_{y\in\R} \nh{u(t)-Q(\cdot-y)}\xrightarrow[t\to+\infty]{} 0$, then there exist a $C^1$ map $x~: t\in\R\mapsto x(t)\in\R$, $t_0\in\R$ and $C>0$ such that \[ \forall t\geq t_0, \quad \nh{u(t,\cdot+x(t))-Q}\leq Ce^{-e_0t}. \]
\end{lem}

\subsubsection{Step 2: Removing modulation}

\begin{itemize}
\item From the previous point, we have in fact $|(e^{e_0t}\alpha_-(t))'| \leq Ce^{-e_0t} \in L^1([t_0,+\infty))$, and so there exists \[ \lim_{t\to+\infty} e^{e_0t}\alpha_-(t) =: A\in\R \] with $|e^{e_0t}\alpha_-(t)-A|\leq Ce^{-e_0t}$ for $t\geq t_0$ by integration. Similarly, since $|x'(t)-1|\leq C\nh{\eps(t)} \leq Ce^{-e_0t}$, then $\exists\lim_{t\to+\infty} x(t)-t =: x_0\in\R$ with $|x(t)-t-x_0|\leq Ce^{-e_0t}$.
\item Now consider the special solution $U^A$ constructed in proposition \ref{th:ua}, defined for a $t_0$ chosen possibly larger, and still write $U^A(t,x+t)=Q(x)+h^A(t,x)$. Let \[ v(t,x)=u(t,x+t+x_0)-Q(x)-h^A(t,x)=u(t,x+t+x_0)-U^A(t,x+t). \] So we want to prove $v=0$ to complete the proof of proposition \ref{th:uniqueness}. We first give estimates on $v$ using the previous estimates on $\eps$.
\item Since $v(t,x)=\eps(t,x-(x(t)-t-x_0))-h^A(t,x)+Q(x-(x(t)-t-x_0))-Q(x)$, then we simply obtain exponential decay for $v$ for $t_0$ large enough, by lemma \ref{th:Qmass} and exponential decay of $h^A$: \begin{align*} \nh{v(t)} &\leq \nh{\eps(t)} + \nh{h^A(t)} +\nh{Q-Q(\cdot-(x(t)-t-x_0))}\\ &\leq Ce^{-e_0t}+C|x(t)-t-x_0| \leq Ce^{-e_0t}. \end{align*}
\item Moreover, we can write \[ u(t,x) = Q(x-x(t))+\eps(t,x-x(t)) = Q(x-t-x_0)+h^A(t,x-t-x_0)+v(t,x-t-x_0). \] If we denote $\omega(t,x)=Q(x-(x(t)-t-x_0))-Q(x)-(x(t)-t-x_0) Q'(x)$, we have $\nli{\omega(t)}\leq C\Carre{x(t)-t-x_0} \leq Ce^{-2e_0t}$ by Taylor-Lagrange inequality, and \[ v(t,x)= (x(t)-t-x_0)Q'(x)-h^A(t,x)+\eps(t,x-(x(t)-t-x_0)) +\omega(t,x). \] Moreover, we have for all $x\in\R$ and $t\geq t_0$: \begin{align*} |\eps(t,x-(x(t)-t-x_0))-\eps(t,x)| &= \left| \int_x^{x-(x(t)-t-x_0)} \partial_x\eps(t,s)\,ds \right|\\ &\leq \sqrt{|x(t)-t-x_0|}\cdot\nh{\eps(t)}\leq Ce^{-\frac{3}{2}e_0t} \end{align*} by the Cauchy-Schwarz inequality. We have finally \begin{equation} \label{eq:transmission} v(t,x)= (x(t)-t-x_0)Q'(x)-h^A(t,x)+\eps(t,x)+\omega(t,x) \end{equation} where $\omega$ verifies $\nli{\omega(t)}\leq Ce^{-\frac{3}{2}e_0t}$.
\item Following the proof (v) in lemma \ref{th:Z}, we now decompose \begin{equation} \label{eq:decomposition} v(t,x) = \alphap(t)\Y_-(x) + \alpham(t)\Y_+(x)+\beta(t)Q'(x)+v_{\bot}(t,x) \end{equation} with \[ \alphap(t) = \int Z_+v(t),\ \alpham(t)=\int Z_-v(t),\ \beta(t) = \nld{Q'}^{-2} \int \Big(v(t)-\alphap(t)\Y_- -\alpham(t)\Y_+\Big)Q'. \] Hence we have $(v_{\bot},Q')=(v_{\bot},Z_+) =(v_{\bot},Z_-)=0$, and so by (iii) of lemma \ref{th:Z}: \begin{equation} \label{eq:vortho} (Lv_{\bot},v_{\bot}) \geq \sigma_1 \nh{v_{\bot}}^2. \end{equation}
\item Multiplying \eqref{eq:transmission} by $Z_{\pm}$, we obtain information on $\alpha_{\pm}^A$. Indeed, since $(Z_{\pm},Q')=0$, then we have \[ \alpha_{\pm}^A = -(h^A,Z_{\pm})+\alpha_{\pm} +(\omega,Z_{\pm}). \] But $|(h^A,Z_+)|\leq Ce^{-2e_0t}$ since $(\Y_+,Z_+)=0$, and $|\alpha_+|\leq Ce^{-2e_0t}$ by \eqref{eq:alphap}, hence $|\alphap|\leq Ce^{-\frac{3}{2}e_0t}$. Similarly, $(\Y_+,Z_-)=1$ implies that $|(h^A,Z_-)-Ae^{-e_0t}|\leq Ce^{-2e_0t}$, and since $|\alpha_--Ae^{-e_0t}|\leq Ce^{-2e_0t}$, we also get $|\alpham|\leq Ce^{-\frac{3}{2}e_0t}$. To sum up this step, we have \eqref{eq:decomposition} with the following estimates for $t\geq t_0$: \begin{equation} \label{eq:alphapm} |\alphap(t)|\leq Ce^{-\frac{3}{2}e_0t},\quad |\alpham(t)|\leq Ce^{-\frac{3}{2}e_0t},\quad \nh{v(t)}\leq Ce^{-e_0t}. \end{equation}
\end{itemize}

In \eqref{eq:alphapm}, it is essential to have obtained estimates better than $Ce^{-e_0t}$ for $\alpha_{\pm}^A$ (see next step).

\subsubsection{Step 3: Exponential decay at any order}

\begin{itemize}
\item We want to prove in this section that $v$ decays exponentially at any order to $0$. In other words, we prove: \begin{equation} \label{eq:toutordre} \forall\gamma>0, \exists C_{\gamma}>0, \forall t\geq t_0,\quad \nh{v(t)}\leq C_{\gamma} e^{-\gamma t}. \end{equation} It has been proved for $\gamma=e_0$, so that it is enough to prove it by induction on $\gamma\geq e_0$: suppose that $\nh{v(t)}\leq Ce^{-\gamma t}$ and let us prove that it implies $\nh{v(t)}\leq C'e^{-\left( \gamma+\frac{1}{2}e_0\right)t}$.
\item Since $u$ and $U^A$ are solutions of \eqref{eq:gKdV}, $v$ verifies the following equation: \begin{equation} \label{eq:v} \partial_t v-\partial_x v+\partial_x^3 v+\partial_x\left[ \Puiss{Q+h^A+v}{p}-\Puiss{Q+h^A}{p} \right]=0. \end{equation} But \begin{align*} \Puiss{Q+h^A+v}{p}-\Puiss{Q+h^A}{p} &= p\Puiss{Q+h^A}{p-1}v +\sum_{k=2}^p \binom{p}{k} \Puiss{Q+h^A}{p-k}v^k\\ &= pQ^{p-1}v+\omega_1(t,x)v+\omega_2(t,x)v^2 \end{align*} where $\omega_1(t,x)=p\left(\sum_{k=1}^{p-1} \binom{p-1}{k} Q^{p-1-k}\Puiss{h^A}{k}\right)$ and $\omega_2(t,x)=\sum_{k=2}^p \binom{p}{k} \Puiss{Q+h^A}{p-k}v^{k-2}$. Since $\nli{h^A(t)}\leq C\nh{h^A(t)}\leq Ce^{-e_0t}$ and $\nli{v(t)}\leq C\nh{v(t)}\leq C$, we have the estimates \begin{equation} \nli{\omega_1(t)}\leq Ce^{-e_0t},\quad \nli{\omega_2(t)} \leq C, \end{equation} and \eqref{eq:v} can be rewritten \begin{equation} \label{eq:v2} \partial_t v+\L v +\partial_x[\omega_1(t,x)v] +\partial_x[\omega_2(t,x)v^2]=0. \end{equation}
\item If we multiply \eqref{eq:v2} by $Z_+$ and integrate, we get ${\alphap}'-e_0\alphap = \int \omega_1vZ'_+ + \int \omega_2v^2Z'_+$, and so \begin{align*} |{\alphap}'-e_0\alphap| &\leq \nli{\omega_1(t)}\nli{v(t)}\nlu{Z'_+} + \nli{\omega_2(t)}\nli{v(t)}^2\nlu{Z'_+}\\ &\leq Ce^{-(\gamma+e_0)t} +Ce^{-2\gamma t} \leq Ce^{-(\gamma+e_0)t}. \end{align*} Consequently, we have $|(e^{-e_0t}\alphap)'|\leq Ce^{-(\gamma+2e_0)t}$, and since $e^{-e_0t}\alphap(t) \xrightarrow[t\to+\infty]{} 0$ by \eqref{eq:alphapm}, we get by integration $|\alphap(t)|\leq Ce^{-(\gamma+e_0)t}$.

    Multiplying \eqref{eq:v2} by $Z_-$, we obtain similarly $|{\alpham}'+e_0\alpham|\leq Ce^{-(\gamma+e_0)t}$, and so $|\alpham(t)|\leq Ce^{-(\gamma+e_0)t}$, since $|e^{e_0t}\alpham(t)|\leq Ce^{-\frac{1}{2}e_0t} \xrightarrow[t\to+\infty]{} 0$ still by \eqref{eq:alphapm}.
\item We now want to estimate $|(Lv,v)|$. To do this, we rewrite \eqref{eq:v} as \[ \partial_t v +\partial_x \left[ \partial_x^2 v -v + \Puiss{Q+h^A+v}{p}-\Puiss{Q+h^A}{p} \right] = 0, \] multiply this equality by the expression in the brackets and integrate, to obtain $\int \partial_t v\cdot \left[ \partial_x^2 v -v + \Puiss{Q+h^A+v}{p}-\Puiss{Q+h^A}{p} \right] = 0$. In other words, if we define \[ F(t) = \frac{1}{2}\int v_x^2 +\frac{1}{2}\int v^2 -\int \frac{1}{p+1}\Puiss{Q+h^A+v}{p+1} +\int v\Puiss{h^A+Q}{p} +\int \frac{1}{p+1} \Puiss{h^A+Q}{p+1}, \] we have: $\displaystyle F'(t) = -\int \partial_t h^A\cdot \left[ \Puiss{Q+h^A+v}{p}-\Puiss{Q+h^A}{p} -pv\Puiss{Q+h^A}{p-1} \right]$.

    But $h^A$ verifies \eqref{eq:linear} by definition, so $\partial_t h^A = -\partial_x^3 h^A+\partial_x h^A - p\partial_x(Q^{p-1}h^A) +R(h^A)$. Moreover, by proposition \ref{th:ua}, there exists $C>0$ such that for all $t\geq t_0$, we have ${\|h^A(t)\|}_{H^4} \leq Ce^{-e_0t}$. We deduce that \[ \ninf{\partial_t h^A}\leq C\nh{\partial_t h^A}\leq C {\| h^A(t)\|}_{H^4} \leq Ce^{-e_0t}. \] Therefore $|F'(t)| \leq C\ninf{\partial_t h^A}\nld{v(t)}^2 \leq Ce^{-(2\gamma+e_0)t}$, and so $|F(t)|\leq Ce^{-(2\gamma+e_0)t}$ by integration, since $\lim_{t\to +\infty} F(t) = 0$. Moreover, by developing $\Puiss{Q+h^A+v}{p+1}$ in the expression of $F$, we get \begin{align*} F(t) &= \frac{1}{2}\int v_x^2 +\frac{1}{2}\int v^2 -\frac{p}{2}\int \Puiss{Q+h^A}{p-1}v^2 -\frac{1}{p+1} \sum_{k=3}^{p+1} \binom{p+1}{k} \int \Puiss{Q+h^A}{p+1-k}v^k\\ &= \frac{1}{2}(Lv,v) - \frac{1}{2} \int \omega_1(t,x)v^2 - \int \wt{\omega_2}(t,x)v^3 \end{align*} where $\omega_1$ defined above and $\wt{\omega_2}(t,x)= \frac{1}{p+1}\sum_{k=3}^{p+1} \binom{p+1}{k} \Puiss{Q+h^A}{p+1-k} v^{k-3}$ verify the estimates $\nli{\omega_1(t)}\leq Ce^{-e_0t}\ \m{ and }\ \nli{\wt{\omega_2}(t)}\leq C$. Hence we have \begin{align*} \left| F(t)-\frac{1}{2}(Lv,v)\right| &\leq \frac{1}{2}\nli{\omega_1(t)}\nld{v(t)}^2 + \nli{\wt{\omega_2}(t)}\nh{v(t)}^3\\ &\leq Ce^{-(2\gamma+e_0)t} +Ce^{-3\gamma t} \leq Ce^{-(2\gamma+e_0)t}. \end{align*} Thus, we finally obtain $|(Lv,v)|\leq Ce^{-(2\gamma+e_0)t}$.
\item The previous points allow us to estimate $\nh{v_{\bot}}$. Indeed, we have by straightforward calculation from \eqref{eq:decomposition} the identity \[ (Lv,v) = (Lv_{\bot},v_{\bot}) +2\alphap\alpham, \] and so $|(Lv_{\bot},v_{\bot})|\leq |(Lv,v)| +2|\alphap|\cdot|\alpham| \leq Ce^{-(2\gamma+e_0)t}+Ce^{-(2\gamma +2e_0)t} \leq Ce^{-(2\gamma+e_0)t}$. But from \eqref{eq:vortho}, we deduce that $\sigma_1\nh{v_{\bot}}^2 \leq Ce^{-(2\gamma+e_0)t}$, and so $\nh{v_{\bot}}\leq Ce^{-(\gamma+\frac{1}{2}e_0)t}$.
\item To conclude this step, it is now enough to estimate $|\beta(t)|$, since the conclusion will then immediately follow from decomposition \eqref{eq:decomposition}. To do this, we first multiply \eqref{eq:v2} by $Q'$ and integrate, so that \begin{align*} |(\partial_t v,Q') +(\L v,Q')| &\leq \nli{\omega_1(t)}\nli{v(t)}\nlu{Q''} + \nli{\omega_2(t)}\nli{v(t)}^2\nlu{Q''}\\ &\leq Ce^{-(\gamma+e_0)t}+Ce^{-2\gamma t} \leq Ce^{-(\gamma+e_0)t}. \end{align*} Moreover, by applying $\L$ to \eqref{eq:decomposition}, we get $\L v = -e_0\alphap\Y_- +e_0\alpham\Y_+ +\L v_{\bot}$, and so \begin{align*} \nld{Q'}^2 \beta'(t) &= (\partial_t v-{\alphap}'\Y_--{\alpham}'\Y_+,Q')\\ &= (\partial_t v+\L v,Q') -(-e_0\alphap\Y_- +e_0\alpham\Y_+ +{\alphap}'\Y_- +{\alpham}'\Y_+,Q') - (\L v_{\bot},Q')\\ &= (\partial_t v+\L v,Q') -({\alphap}'-e_0\alphap)(\Y_-,Q') -({\alpham}'+e_0\alpham)(\Y_+,Q') + (v_{\bot},LQ''). \end{align*} Finally, we obtain thanks to all previous estimates: \begin{align*} |\beta'(t)| &\leq C|(\partial_t v+\L v,Q')| +C|{\alphap}'-e_0\alphap| +C|{\alpham}'+e_0\alpham| +C\nld{v_{\bot}}\\ &\leq Ce^{-(\gamma+e_0)t} + Ce^{-(\gamma+e_0)t} + Ce^{-(\gamma+e_0)t} + Ce^{-(\gamma+\frac{1}{2}e_0)t} \leq Ce^{-(\gamma+\frac{1}{2}e_0)t} \end{align*} and so $|\beta(t)|\leq Ce^{-(\gamma+\frac{1}{2}e_0)t}$ by integration.
\end{itemize}

\subsubsection{Step 4: Conclusion of uniqueness argument by contraction}

\begin{itemize}
\item The final argument, which corresponds to step 3 in \cite{duy}, is an argument of contraction in short time. In other words, we want to reproduce the contraction argument developed in section \ref{subsubsec:construction} on a short interval of time, with suitable norms.

    Define $w(t,x)=v(t,x-t)$, so that \eqref{eq:v} can be rewritten \[ \partial_t w +\partial_x^3 w = -\partial_x \left[ \Puiss{Q(x-t)+h^A(t,x-t)+w}{p} - \Puiss{Q(x-t)+h^A(t,x-t)}{p} \right]. \] If we denote $\Omega_w(t,x) = \sum_{k=1}^p \binom{p}{k} \Puiss{Q(x-t)+h^A(t,x-t)}{p-k} w^k(t,x)$, then the equation on $w$ can be rewritten \[ \partial_t w+\partial_x^3 w = -\partial_x(\Omega_w). \] Moreover, we have by previous steps: $\forall \gamma>0, \exists C_{\gamma}>0, \forall t\geq t_0,\quad \nh{w(t)}\leq C_{\gamma} e^{-\gamma t}$.
\item Now let $t_1\geq t_0$, $\tau>0$ to fix later, and $I=(t_1,t_1+\tau)$. Moreover, consider the non-linear equation in $\wt{w}$: \begin{equation} \label{eq:wt} \left\{ \begin{aligned} &\partial_t \wt{w} +\partial_x^3 \wt{w} = -\partial_x(\Omega_{\tilde{w}}),\\ &\wt{w}(t_1+\tau)=w(t_1+\tau). \end{aligned} \right. \end{equation} Note that $w$ is of course a solution of \eqref{eq:wt}, associated to a solution $u$ of \eqref{eq:gKdV} in the sense of \cite{kpv}.
\item Then for $t\in I$, we have the following Duhamel's formula: \[ \wt{w}(t) = \M^I(\wt{w})(t) := W(t-t_1-\tau)w(t_1+\tau) + \int_t^{t_1+\tau} W(t-t')\partial_x[\Omega_{\tilde{w}}(t')]\,dt'. \]  Similarly as in section \ref{subsubsec:construction}, we consider \[ \left\{ \begin{aligned} N_1^I(\wt{w}) &=\sup_{t\in I} \nh{\wt{w}(t)},\ N_2^I(\wt{w}) = {\| \wt{w}\|}_{L_x^5L_I^{10}} + {\| \partial_x \wt{w}\|}_{L_x^5L_I^{10}},\\ \Lambda^I(\wt{w}) &= \max(N_1^I(\wt{w}),N_2^I(\wt{w})), \end{aligned} \right. \] and we prove that for $t_1$ large enough, $\tau$ small enough independently of $t_1$, and $K>1$ to determine, $\wt{w} \mapsto \M^I(\wt{w})$ is a contraction on \[ B = \{ \wt{w}\in C^0(I,H^1)\ |\ \Lambda^I(\wt{w})\leq 3K\nh{w(t_1+\tau)} \}. \]

    In other words, we want to estimate $\Lambda^I(\M^I(\wt w))$ in terms of $\Lambda^I(\wt w)$, and as in section \ref{subsubsec:construction}, we estimate only the term \[ \partial_x \M^I(\wt{w})(t) = W(t-t_1-\tau) \partial_x w(t_1+\tau) + \drondx \int_t^{t_1+\tau} W(t-t')\partial_x[\Omega_{\tilde{w}}(t')]\,dt' \] in $L_I^{\infty}L_x^2$ and $L_x^5L_I^{10}$ norms. The term $\M^I(\wt w)(t)$ is treated similarly.
\item Firstly, for the linear term, we have \[ \left\{ \begin{array}{l} \nld{W(t-t_1-\tau)\partial_x w(t_1+\tau)} = \nld{\partial_x w(t_1+\tau)} \leq \nh{w(t_1+\tau)},\\ {\| W(t-t_1-\tau)\partial_x w(t_1+\tau)\|}_{L_x^5L_I^{10}} \leq C\nld{\partial_x w(t_1+\tau)} \leq C\nh{w(t_1+\tau)}, \end{array} \right. \] since $W$ is unitary on $L^2$ and by the linear estimate (2.3) of \cite{kpvcritical}: ${\| W(t)u_0\|}_{L_x^5L_t^{10}}\leq C\nld{u_0}$.
\item For the non linear term, we have to use estimates similar to \eqref{eq:kpv1} and \eqref{eq:kpv2}. We obtain easily by a similar proof that for all $g\in L_x^1L_I^2$, \[ {\left\| \drondx \int_t^{t_1+\tau} W(t-t')g(x,t')\,dt' \right\|}_{L_I^{\infty}L_x^2} + {\left\| \drondx \int_t^{t_1+\tau} W(t-t')g(x,t')\,dt' \right\|}_{L_x^5L_I^{10}} \leq C{\|g\|}_{L_x^1L_I^2}. \] Hence we get \[ \left\{ \begin{array}{l} {\left\| \drondx \int_t^{t_1+\tau} W(t-t')\partial_x[\Omega_{\tilde{w}}(t')]\,dt' \right\|}_{L_I^{\infty}L_x^2} \leq C{\|\partial_x(\Omega_{\tilde{w}})\|}_{L_x^1L_I^2},\\ {\left\| \drondx \int_t^{t_1+\tau} W(t-t')\partial_x[\Omega_{\tilde{w}}(t')]\,dt' \right\|}_{L_x^5L_I^{10}} \leq C{\|\partial_x(\Omega_{\tilde{w}})\|}_{L_x^1L_I^2}. \end{array} \right. \] We deduce that we only have to estimate ${\|\partial_x(\Omega_{\tilde{w}})\|}_{L_x^1L_I^2}$. There are many terms to estimate, so as in section \ref{subsubsec:construction}, we only treat three typical terms: $\mathbf{A} = {\| \partial_x \wt w\cdot {\wt w}^4\cdot {\wt w}^{p-5}\|}_{L_x^1L_I^2}$, $\mathbf{B} = {\| \partial_x \wt w\cdot\Puiss{h^A}{p-1} (t,x-t)\|}_{L_x^1L_I^2}$, $\mathbf{D} = {\| \partial_x \wt w\cdot Q^{p-1}(x-t)\|}_{L_x^1L_I^2}$.

    For $\mathbf{A}$, we have by Hölder's inequality: \[ \mathbf{A} \leq {\| \wt w\|}_{L_I^{\infty}L_x^{\infty}}^{p-5} {\|\partial_x \wt w\|}_{L_x^5L_I^{10}} {\|\wt w\|}_{L_x^5L_I^{10}}^4 \leq Ce^{-e_0t_1}{N_2^I(\wt w)}^5 \leq C'e^{-e_0t_1} N_2^I(\wt w). \] Indeed, we have \[ \Lambda^I(\wt w) \leq 3K\nh{w(t_1+\tau)} \leq Ce^{-e_0t_1}\leq 1\] for $t_1$ large enough, by exponential decay of $w$ in $H^1$. In particular, we have $N_2^I(\wt w)\leq 1$ and ${\| \wt w\|}_{L_I^{\infty}L_x^{\infty}}^{p-5} \leq C{N_1^I(\wt w)}^{p-5} \leq Ce^{-e_0t_1}$ since $p-5\geq 1$.

    For $\mathbf{B}$, we write similarly \[ \mathbf{B} \leq {\|h^A\|}_{L_I^{\infty}L_x^{\infty}}^{p-5} {\|\partial_x \wt w\|}_{L_x^5L_I^{10}} {\|h^A(t,x-t)\|}_{L_x^5L_I^{10}}^4. \] Moreover, we have by construction of $h^A$ (see section \ref{subsubsec:construction}), ${\|h^A\|}_{L_I^{\infty}L_x^{\infty}}^{p-5} \leq Ce^{-e_0t_1}$ since $\nh{h^A(t)}\leq Ce^{-e_0t}\leq Ce^{-e_0t_1}$ for $t\geq t_1$ and $p-5\geq 1$, and \begin{align*} {\| h^A(t,x-t)\|}_{L_x^5L_I^{10}} &\leq {\|h^A(t,x-t)\|}_{L_x^5L_{[t_1,+\infty)}^{10}}\\ &\leq {\|(h^A-\V_{k_0}^A)(t,x-t)\|}_{L_x^5L_{[t_1,+\infty)}^{10}} + {\|V_{k_0}^A(t,x-t)\|}_{L_x^5 L_{[t_1,+\infty)}^{10}}\\ &\leq Ce^{-(k_0+\frac{1}{2})e_0t_1}+Ce^{-e_0t_1} \leq Ce^{-e_0t_1}. \end{align*} Note that the estimate ${\|V_{k_0}^A(t,x-t)\|}_{L_x^5 L_{[t_1,+\infty)}^{10}} \leq Ce^{-e_0t_1}$ follows from the paragraph on $\mathbf{II}_{p-2,1,1}$ in section \ref{subsubsec:construction}.

    For $\mathbf{D}$, we use exponential decay of $Q$ to write \begin{multline*} \mathbf{D} \leq C\int_{\R} \sqrt{ \int_I e^{-2|x-t|} \Carre{\partial_x \wt w}\,dt} dx \leq C\int_{-\infty}^{t_1} e^x \sqrt{\int_I e^{-2t} \Carre{\partial_x \wt w}\,dt} dx\\  +C\int_{t_1+\tau}^{+\infty} e^{-x} \sqrt{\int_I e^{2t} \Carre{\partial_x \wt w}\,dt} dx + C\int_I \sqrt{\int_I \Carre{\partial_x \wt w}\,dt} dx = \mathbf{D_1}+\mathbf{D_2}+\mathbf{D_3}. \end{multline*} But by the Cauchy-Schwarz inequality, we get \[ \left\{ \begin{aligned} \mathbf{D_1} &\leq Ce^{t_1} \sqrt{ \int_I e^{-2t} \int_{\R} \Carre{\partial_x \wt w}\,dx\,dt} \leq Ce^{t_1}N_1^I(\wt w)\sqrt{\int_I e^{-2t}\,dt} \leq C\sqrt{\tau}N_1^I(\wt w),\\ \mathbf{D_2} &\leq Ce^{-(t_1+\tau)}\sqrt{\int_I e^{2t} \int_{\R} \Carre{\partial_x \wt w}\,dx\,dt} \leq Ce^{-(t_1+\tau)}N_1^I(\wt w) \sqrt{\int_I e^{2t}\,dt} \leq C\sqrt{\tau}N_1^I(\wt w),\\ \mathbf{D_3} &\leq C\sqrt{\tau} \sqrt{\int_I \int_I \Carre{\partial_x \wt w}\,dx\,dt} \leq C\tau N_1^I(\wt w). \end{aligned} \right. \] Hence we obtain $\mathbf{D} \leq C\sqrt{\tau}N_1^I(\wt w)$.
\item In conclusion, we have shown that there exist $K,C_1,C_2>0$ such that \[ \Lambda^I(\M^I(\wt w)) \leq K \Big[ \nh{w(t_1+\tau)} +C_1e^{-e_0t_1}\Lambda^I(\wt w) +C_2\sqrt{\tau}\Lambda^I(\wt w) \Big]. \] Now fix $\tau = \frac{1}{9C_2^2K^2}$ and $t_1$ such that $C_1e^{-e_0t_1}\leq \frac{1}{3K}$, thus we get \[ \Lambda^I(\M^I(\wt w)) \leq K\nh{w(t_1+\tau)} +\frac{2}{3}\Lambda^I(\wt w). \] We conclude that $\M^I$ maps $B$ into itself for this choice of $t_1,\tau,K$. We prove similarly that $\M^I$ is a contraction on $B$, and so there exists a unique solution $\wt w \in B$ of \eqref{eq:wt}.
\item Now we identify $w$ and $\wt w$. It is well known for \eqref{eq:gKdV} that for regular solutions ($H^2$), uniqueness holds by energy method. Since $w$ and $\wt w$ are both obtained by fixed point, we get $w=\wt w$ by continuous dependence, persistence of regularity and density. In particular, $w\in B$, and so \[ \nh{w(t_1)} \leq N_1^I(w)\leq \Lambda^I(w)\leq 3K\nh{w(t_1+\tau)}. \] To conclude the proof, we fix $t\geq t_1$, and we remark that a simple iteration argument and the exponential decay at any order of $w$ show that for all $n\in\N$, we have \[ \nh{w(t)} \leq {(3K)}^n \nh{w(t+n\tau)} \leq C_{\gamma}{(3K)}^n e^{-\gamma t}e^{-\gamma n\tau} = C_{\gamma}e^{-\gamma t}\Puiss{3Ke^{-\gamma\tau}}{n}.\] We finally choose $\gamma$ large enough so that $3Ke^{-\gamma \tau}\leq \frac{1}{2}$. Thus, \[ \nh{w(t)} \leq \frac{C}{2^n} \xrightarrow[n\to+\infty]{} 0, \] \emph{i.e.} $\nh{w(t)}=0$. This finishes the proof of proposition \ref{th:uniqueness}.\bigskip
\end{itemize}

\subsection{Corollaries and remarks} \label{subsec:remarks}

\begin{cor} \label{th:uac}
Let $c>0$.
\begin{enumerate}
\item There exists a one-parameter family ${(U_c^A)}_{A\in\R}$ of solutions of \eqref{eq:gKdV} such that \[ \forall A\in\R, \exists t_0\in\R, \forall s\in\R, \exists C>0, \forall t\geq t_0,\quad {\|U_c^A(t,\cdot+ct)-Q_c\|}_{H^s} \leq Ce^{-e_0c^{3/2}t}. \]
\item If $u_c$ is a solution of \eqref{eq:gKdV} such that $\lim_{t\to+\infty} \inf_{y\in\R} \nh{u_c(t)-Q_c(\cdot-y)}=0$, then there exist $A\in\R$, $t_0\in\R$ and $x_0\in\R$ such that $u_c(t)=U_c^A(t,\cdot-x_0)$ for $t\geq t_0$.
\end{enumerate}
\end{cor}

\begin{proof}
The proof, based on the scaling invariance, is very similar to the proof of corollary \ref{th:scaling}. We recall that if $u(t,x)$ is a solution of \eqref{eq:gKdV}, then $\lambda^{\frac{2}{p-1}}u(\lambda^3t,\lambda x)$ with $\lambda>0$ is also a solution.
\begin{enumerate}
\item We define $U_c^A$ by $U_c^A(t,x) = c^{\frac{1}{p-1}}U^A(c^{3/2}t,\sqrt cx)$, where $U^A$ is defined in theorem \ref{th:main}. Since $U^A(c^{3/2}t,\sqrt cx+c^{3/2}t) = Q(\sqrt cx) + Ae^{-e_0c^{3/2}t}\Y_+(\sqrt cx)+O(e^{-2e_0c^{3/2}t})$ and $Q_c(x) = c^{\frac{1}{p-1}}Q(\sqrt cx)$, then $U_c^A$ satisfies \[ U_c^A(t,x+ct) = Q_c(x)+Ac^{\frac{1}{p-1}}e^{-e_0c^{3/2}t}\Y_+(\sqrt cx)+O(e^{-2e_0c^{3/2}t}). \]
\item Let $u$ be the solution of \eqref{eq:gKdV} defined by $u(t,x)=c^{-\frac{1}{p-1}} u_c\left(\frac{t}{c^{3/2}},\frac{x}{\sqrt c}\right)$. Then we have \[ u(t,x)-Q(x-y)=c^{-\frac{1}{p-1}}u_c\left(\frac{t}{c^{3/2}},\frac{x}{\sqrt c}\right)-c^{-\frac{1}{p-1}} Q_c\left(\frac{x-y}{\sqrt c}\right) \] for all $y\in\R$, and so like in the proof of corollary \ref{th:scaling}, \[ \inf_{y\in\R} \nh{u(t)-Q(\cdot-y)} \leq K(c) \inf_{y\in\R} \nh{u_c\left( \frac{t}{c^{3/2}}\right) -Q_c\left(\cdot -\frac{y}{\sqrt c}\right)} \xrightarrow[t\to+\infty]{} 0. \] Therefore by theorem \ref{th:main}, there exist $A\in\R$ and $x_0\in\R$ such that $u(t,x)=U^A(t,x-x_0)$, and so finally $u_c(t,x) = U_c^A\left( t,x-\frac{x_0}{\sqrt c}\right)$. \qedhere
\end{enumerate}
\end{proof}

\begin{prop} \label{th:troistypes}
Up to translations in time and in space, there are only three special solutions: $U^1$, $U^{-1}$ and $Q$. More precisely, one has (for $t$ large enough in each case):
\begin{enumerate}[(a)]
\item If $A>0$, then $U^A(t)=U^1(t+t_A,\cdot+t_A)$ for some $t_A\in\R$.
\item If $A=0$, then $U^0(t)=Q(\cdot-t)$.
\item If $A<0$, then $U^A(t)=U^{-1}(t+t_A,\cdot+t_A)$ for some $t_A\in\R$.
\end{enumerate}
\end{prop}

\begin{proof}
\begin{enumerate}[(a)]
\item Let $A>0$ and denote $t_A = -\frac{\ln A}{e_0}$. Then by proposition \ref{th:ua}, \[ U^1(t+t_A,x+t+t_A) = Q(x)+e^{-e_0(t+t_A)}\Y_+(x)+O(e^{-2e_0t}) = Q(x)+Ae^{-e_0t}\Y_+(x)+O(e^{-2e_0t}). \] In particular, we have $\lim_{t\to+\infty} \inf_{y\in\R} \nh{U^1(t+t_A)-Q(\cdot-y)}=0$, and so by proposition \ref{th:uniqueness}, there exist $\wt{A}\in\R$ and $x_0\in\R$ such that $U^1(t+t_A)=U^{\wt A}(t,\cdot-x_0)$. But still by proposition \ref{th:ua}, we have $U^1(t+t_A,x+t+t_A)=U^{\wt A}(t,x+t+t_A-x_0) = Q(x+t_A-x_0) +\wt A e^{-e_0t}\Y_+(x+t_A-x_0)+O(e^{-2e_0t})$, and so \[ Q(x+t_A-x_0)+\wt Ae^{-e_0t}\Y_+(x+t_A-x_0) +O(e^{-2e_0t}) = Q(x)+Ae^{-e_0t}\Y_+(x)+O(e^{-2e_0t}). \] The first order imposes $x_0=t_A$, since $\nh{Q-Q(\cdot+t_A-x_0)}\leq Ce^{-e_0t}$ and so lemma \ref{th:Qmass} applies for $t$ large. Similarly, the second order imposes $\wt A=A$, as expected.
\item Since $\inf_{y\in\R} \nh{Q(\cdot-t)-Q(\cdot-y)}=0$, then proposition \ref{th:uniqueness} applies, so there exist $A\in\R$ and $x_0\in\R$ such that $Q(x-t)=U^A(t,x-x_0)$. Hence we have by proposition \ref{th:ua} \[ U^A(t,x+t) = Q(x-x_0) = Q(x)+Ae^{e_0t}\Y_+(x)+O(e^{-2e_0t}). \] As in the previous case, it follows first that $x_0=0$, then $A=0$, and so the result.
\item For $A<0$, the proof is exactly the same as $A>0$, with $-A$ instead of $A$. \qedhere\bigskip
\end{enumerate}
\end{proof}

We conclude this paper by two remarks, based on the following claim. The first one is the fact that $U^{-1}(t)$ is defined for all $t\in\R$, and the second one is the identification of the special solution $w(t)$ constructed in section \ref{sec:compactness} among the family $(U^A)$ constructed in section \ref{sec:contraction}.

\begin{claim} \label{th:gradient}
For all $c>0$, $\nld{\partial_x U_c^A(t)}^2-\nld{Q'_c}^2$ has the sign of $A$ as long as $U_c^A(t)$ exists.
\end{claim}

\begin{proof}
\begin{itemize}
\item From corollary \ref{th:uac}, we have \[ \partial_x U_c^A(t,x+ct) = Q'_c(x) + Ac^{\frac{p+1}{2(p-1)}}e^{-e_0c^{3/2}t}\Y'_+(\sqrt cx) +O(e^{-2e_0c^{3/2}t}) \] and so \[ \nld{\partial_x U_c^A(t)}^2 - \nld{Q'_c}^2 = 2Ac^{\frac{p+1}{2(p-1)}}e^{-e_0c^{3/2}t} \int Q'_c(x)\Y'_+(\sqrt cx)\,dx +O(e^{-2e_0c^{3/2}t}). \] But $\int Q'_c(x)\Y'_+(\sqrt cx)\,dx = c^{\frac{1}{p-1}}\int Q'(y)\Y'_+(y)\,dy>0$ by the substitution $y=\sqrt cx$ and the normalization chosen in lemma \ref{th:Z}, and so $\nld{\partial_x U_c^A(t)}^2 - \nld{Q'_c}^2$ has the sign of $A$ for $t$ large enough.
\item It remains to show that this fact holds as long as $U_c^A(t)$ exists. For example, suppose that $A>0$ and so $\nld{\partial_x U_c^A(t)}^2 - \nld{Q'_c}^2 >0$ for $t\geq t_1$, and suppose for the sake of contradiction that there exists $T<t_1$ such that $U^A(T)$ is defined and $\nld{\partial_x U_c^A(T)}^2 = \nld{Q'_c}^2$. Since $\nh{U_c^A(t,\cdot+ct)-Q_c}\vers 0$, then by \eqref{eq:mass} and \eqref{eq:energy}, we also have $\nld{U_c^A(T)}=\nld{Q_c}$ and $E(U_c^A(T))=E(Q_c)$. In other words, we would get by scaling \[ \nld{U^A(T)}=\nld{Q},\ \nld{\partial_x U^A(T)}=\nld{Q'}\ \m{ and }\ E(U^A(T))=E(Q). \] But the two last identities give in particular $\int {U^A(T)}^{p+1} = \int Q^{p+1}$, and so by \eqref{eq:caracGN} \[ {\|U^A(T)\|}_{L^{p+1}}^{p+1} \geq {\| Q\|}_{L^{p+1}}^{p+1} = C\gn(p)\nld{Q'}^{\frac{p-1}{2}} \nld{Q}^{\frac{p+3}{2}} = C\gn(p)\nld{\partial_x U^A(T)}^{\frac{p-1}{2}} \nld{U^A(T)}^{\frac{p+3}{2}}. \] Still by \eqref{eq:caracGN}, we get $(\lambda_0,a_0,b_0)\in \R_+^* \times\R\times\R$ such that $U^A(T,x)=a_0Q(\lambda_0 x+b_0)$. But $\nld{U^A(T)}=\nld{Q}$ and $\nld{\partial_x U^A(T)}=\nld{Q'}$ impose $\lambda_0=1$ and $a_0\in \{-1,1\}$. Thus, by uniqueness in \eqref{eq:gKdV}, $U^A(t,x)=\pm Q(x-t+T+b_0)$ for all $t\geq T$. In particular, $\nld{\partial_x U_c^A(t)}^2 = \nld{Q'_c}^2$ for $t\geq t_1$, which is a contradiction. The cases $A=0$ and $A<0$ are treated similarly. \qedhere
\end{itemize}
\end{proof}

\begin{rem}
Let us now notice that $U^{-1}$ is globally defined, \emph{i.e.} $U^{-1}(t)$ exists for all $t\in\R$. By the blow up criterion and the mass conservation, it is enough to remark that $\nld{\partial_x U^{-1}(t)}$ is bounded uniformly on its interval of existence, which is an immediate consequence of claim \ref{th:gradient} since $\nld{\partial_x U^{-1}(t)} < \nld{Q'}$ for all $t$.
\end{rem}

\begin{rem}
As noticed in remark \ref{th:choixuon}, we can chose $\lambda_n=1-\frac{1}{n}$ in the definition of $u_{0,n}$ in section \ref{sec:compactness}. We still call $w(t)$ the special solution obtained by this method for this new initial data. In this remark, we prove that $w=U_{c_+}^{-1}$ up to translations in time and in space. We do not know if $U^1$ can be obtained similarly by a compactness method. We recall that $u_{0,n}(x) = \lambda_n Q(\lambda_n^2 x)$, $u_n(T_n,\cdot+x_n(T_n))\cvf \check{w}_0 \neq Q_{c_+}$ and $\nh{w(t,\cdot+\rho(t))-Q_{c_+}}\vers 0$.
\begin{itemize}
\item First note that $\int u_{0,n}'^2 = \lambda_n^4\int Q'^2 < \int Q'^2$ for $n\geq 2$, and let us prove that $\nld{\partial_x(u_n(T_n))}<\nld{Q'}$ for $n$ large enough. Otherwise, there would exist $n$ large and $T\in [0,T_n]$ such that $\nld{\partial_x(u_n(T))}=\nld{Q'}$ and $E(u_{0,n})<E(Q)$. But we have by \eqref{eq:energy}, \begin{align*} E(u_{0,n}) &= E(u_n(T)) = \frac{1}{2}\int \Carre{\partial_x(u_n(T))} -\frac{1}{p+1}\int u_n^{p+1}(T) = \frac{1}{2}\int Q'^2 -\frac{1}{p+1}\int u_n^{p+1}(T)\\ &< E(Q) = \frac{1}{2}\int Q'^2 -\frac{1}{p+1}\int Q^{p+1}. \end{align*} Hence, as $\nld{u_n(T)} = \nld{u_{0,n}} = \nld{Q}$ by \eqref{eq:mass}, \begin{align*} {\|u_n(T)\|}_{L^{p+1}}^{p+1} \geq \int u_n^{p+1}(T) > \int Q^{p+1} &= C\gn(p) \Puiss{\int Q'^2}{\frac{p-1}{4}}\Puiss{\int Q^2}{\frac{p+3}{4}}\\ &= C\gn(p) \Puiss{\int \Carre{\partial_x(u_n(T))}}{\frac{p-1}{4}}\Puiss{\int u_n^2(T)}{\frac{p+3}{4}}, \end{align*} which would be a contradiction with the Gagliardo-Nirenberg inequality \eqref{eq:GN}.
\item Since $u_n(T_n,\cdot+x_n(T_n))\cvf \check{w}_0$ in $H^1$, we obtain $\nld{w'_0}\leq \nld{Q'}$ and $\nld{w_0}\leq \nld{Q}$ by weak convergence. But $\nh{w(t,\cdot+\rho(t))-Q_{c_+}}\vers 0$ implies by \eqref{eq:mass} and \eqref{eq:Qcmass} that \[ \nld{w_0}^2 = \nld{w(t)}^2 = \nld{Q_{c_+}}^2 = c_+^{\frac{5-p}{2(p-1)}} \nld{Q}^2 \leq \nld{Q}^2, \] thus $c_+\geq 1$, and so $\nld{w'_0}^2\leq \nld{Q'}^2 = c_+^{-\frac{p+3}{2(p-1)}}{\|Q'_{c_+}\|}_{L^2}^2 \leq {\|Q'_{c_+}\|}_{L^2}^2$ by \eqref{eq:Qcmass}.
\item Finally, since $\nh{w(t,\cdot+\rho(t))-Q_{c_+}}\vers 0$, corollary \ref{th:uac} applies, and so there exists $A\in\R$ such that $w=U_{c_+}^A$ up to a translation in space. But the conclusion of the previous point and claim \ref{th:gradient} impose $A<0$ (note that $A\neq 0$ since $w_0\neq Q_{c_+}$), \emph{i.e.} $w=U_{c_+}^{-1}$ up to translations in time and in space by proposition \ref{th:troistypes}.
\end{itemize}
\end{rem}

\bibliographystyle{plain}

\end{document}